\documentclass[reqno]{amsart}
\usepackage{amssymb,amsmath,hyperref}
\usepackage{amsrefs}
\usepackage[foot]{amsaddr}
\usepackage{bbold,stackrel}
 
\newcommand{\Z}{{\textsf{\textup{Z}}}}

\newtheorem{thm}{Theorem}

\newtheorem{cor}[thm]{Corollary}
\newtheorem{defi}[thm]{Definition}
\newtheorem{rem}[thm]{Remark}
\newtheorem{nota}[thm]{Notation}

\newtheorem{ack}[thm]{Acknowledgement}

\newtheorem*{tempo*}{Template}

\newcommand\be{\begin{equation}}
\newcommand\ee{\end{equation}} 

\usepackage{amsmath,amsfonts} 

\usepackage[applemac]{inputenc}

\def\bdefi{\begin{defi}\rm}
\def\edefi{\end{defi}}
\def\bnota{\begin{nota}\rm}
\def\enota{\end{nota}}
\def\FIVE{\Pi_{1}^{1}\text{-\textup{\textsf{CA}}}_{0}}

\def\SIXko{\Pi_{k+1}^{1}\text{-\textsf{\textup{CA}}}_{0}}
\def\SIXK{\Pi_{k}^{1}\text{-\textsf{\textup{CA}}}_{0}^{\omega}}

\def\ATR{\textup{\textsf{ATR}}}
\def\ZFC{\textup{\textsf{ZFC}}}
\def\ZF{\textup{\textsf{ZF}}}

\def\L{\textsf{\textup{L}}}

 \def\r{\mathbb{r}}

\def\RCA{\textup{\textsf{RCA}}}
\def\({\textup{(}}
\def\){\textup{)}}

\def\c{\textup{\textsf{c}}}

\def\RCAo{\textup{\textsf{RCA}}_{0}^{\omega}}
\def\ACAo{\textup{\textsf{ACA}}_{0}^{\omega}}
\def\WKL{\textup{\textsf{WKL}}}

\def\WWKL{\textup{\textsf{WWKL}}}
\def\bye{\end{document}}
\def\N{{\mathbb  N}}
\def\Q{{\mathbb  Q}}
\def\R{{\mathbb  R}}

\def\SS{\textup{\textsf{S}}}

\def\di{\rightarrow}

\def\asa{\leftrightarrow}

\def\ACA{\textup{\textsf{ACA}}}

\def\QFAC{\textup{\textsf{QF-AC}}}
\def\PRA{\textup{\textsf{PRA}}}

\def\PERMU{\textup{\textsf{PERM}}}
\def\PERM{\textup{\textsf{PERM}}}

\def\AS{\textup{\textsf{AS}}}
\def\b{\mathbb{b}}
\def\IND{\textup{\textsf{IND}}}
\def\NFP{\textup{\textsf{NFP}}}

\def\HBU{\textup{\textsf{HBU}}}

\def\CAU{\textup{\textsf{CAU}}}

\def\SUB{\textup{\textsf{SUB}}}
\def\ARZ{\textup{\textsf{ARZ}}}
\def\DIN{\textup{\textsf{DIN}}}

\def\net{\textup{\textsf{net}}}

\def\BW{\textup{\textsf{BW}}}

\def\LIND{\textup{\textsf{LIN}}}
\def\LIN{\textup{\textsf{LIN}}}

\def\COH{\textup{\textsf{COH}}}

\def\MCT{\textup{\textsf{MCT}}}

\def\SJ{\mathbb{S}}

\def\eps{\varepsilon}

\def\ADS{\textup{\textsf{ADS}}}

\def\ECF{\textup{\textsf{ECF}}}

\usepackage{graphicx}
\usepackage{tikz}
\usetikzlibrary{matrix, shapes.misc}

\numberwithin{equation}{section}
\numberwithin{thm}{section}

\usepackage{comment}

\begin{document}
\title[Nets and Reverse Mathematics]{Nets and Reverse Mathematics, \\ \small a pilot study}
\author{Sam Sanders}
\address{Department of Mathematics, TU Darmstadt, Germany}
\email{sasander@me.com}
\subjclass[2010]{03B30, 03D65, 03F35}
\keywords{reverse mathematics, higher-order computability theory, nets, Moore-Smith sequences}
\begin{abstract}
Nets are generalisations of sequences involving possibly \emph{uncountable} index sets; this notion was introduced about a century ago by Moore and Smith.    
They also established the generalisation to nets of various basic theorems of analysis due to Bolzano-Weierstrass, Dini, Arzel\`a, and others.
More recently, nets are central to the development of \emph{domain theory}, providing intuitive definitions of the associated Scott and Lawson topologies, among others.  
This paper deals with the Reverse Mathematics study of basic theorems about nets.  
We restrict ourselves to nets indexed by subsets of Baire space, and therefore third-order arithmetic, as such nets suffice to obtain our main results.  
Over Kohlenbach's base theory of \emph{higher-order} Reverse Mathematics, the Bolzano-Weierstrass theorem for nets implies the Heine-Borel theorem \emph{for uncountable covers}.  
We establish similar results for other basic theorems about nets and even 
some equivalences, e.g.\ for Dini's theorem for nets.  
Finally, we show that replacing nets by sequences is hard, but that replacing sequences by nets can obviate the need for the Axiom of Choice, a foundational concern in domain theory.    
In an appendix, we study the power of more general index sets, establishing that the `size' of a net is directly proportional to the power of the associated convergence theorem. 
\end{abstract}

%\setcounter{page}{0}
%\tableofcontents
%\thispagestyle{empty}
%\newpage

\maketitle
\thispagestyle{empty}

\section{Aim and motivation}\label{schintro}
\subsection{Introduction}\label{intro}
The move to more abstract mathematics can be quite concrete and specific: E.\ H.\ Moore presented a framework called \emph{General Analysis} at the 1908 ICM in Rome (\cite{mooreICM}) that was to be a `unifying abstract theory' for various parts of analysis.  
For instance, Moore's framework captures various limit notions in one abstract concept (\cites{moorelimit1}).  
This theory also included a generalisation of the concept of \emph{sequence} to possibly uncountable index sets, nowadays called \emph{nets} or \emph{Moore-Smith sequences}.  
These were first described in \cite{moorelimit2} and then formally introduced by Moore and Smith in \cites{moorsmidje}. 
They also established the generalisation to nets of various basic theorems due to Bolzano-Weierstrass, Dini, and Arzel\`a (\cite{moorsmidje}*{\S8-9}).
More recently, nets are central to the development of \emph{domain theory} (see \cites{gieren, gieren2,degou}), including a definition of the Scott and Lawson topologies in terms of nets.    
Moreover, sequences cannot be used in this context, as expressed in a number of places:
\begin{quote}
Turning to foundations, we feel that the necessity to choose chains where directed subsets are naturally available (such as in function spaces) and thus to rely on the Axiom of Choice without need, is a serious stain on this approach. (\cite{aju}*{\S2.2.4}).
\end{quote}
\begin{quote}
[\dots] clinging to ascending sequences would produce a mathematical theory that becomes rather bizarre, whence our move to directed\footnote{Nets can have uncountable index sets, and the latter are called \emph{directed sets}.} families. (\cite{degou}*{p.\ 59})
\end{quote}
Thus, nets enjoy a rich history, as well as a mainstream (and essential) status in mathematics and computer science.  
Motivated by the above, this paper deals with the study of nets in \emph{Reverse Mathematics} (RM hereafter); the latter program is briefly introduced in Section \ref{prelim}.
Since uncountable index sets are first-class citizens in the theory of nets, we work in Kohlenbach's \emph{higher-order} RM (see Section \ref{prelim1}).  The exact formalisation of nets in higher-order RM is detailed in Definition \ref{strijker} and Section \ref{intronet}.  
In the main part of this paper, we restrict ourselves to nets indexed by subsets of Baire space, i.e.\ part of third-order arithmetic, as such nets are already general enough to obtain our main results.  
More motivation for the RM-study of nets is provided in Section \ref{motiv}, and we summarise our results in Section \ref{sumsum}.  % paper.

\subsection{Summary of results}\label{sumsum}
%We first provide more motivation for our study of nets.  % for our study, based on the following arguments.
First of all, the Bolzano-Weierstrass theorem \emph{for nets} implies both the sequential and \emph{uncountable} open-cover compactness of $[0,1]$.
The latter notion is captured by $\HBU$ (see Section~\ref{prelim2}) and the minimal\footnote{In classical RM, the sequential compactness of $[0,1]$ is equivalent to $\ACA_{0}$ by \cite{simpson2}*{III.2.2}, while the (countable) open-cover compactness of the unit interval is equivalent to $\WKL_{0}$ by \cite{simpson2}*{IV.1}.  In higher-order RM, the open-cover compactness for \emph{uncountable} covers of the unit interval, called $\HBU$, cannot be proved in $\SIXK+\QFAC^{0,1}$ by \cite{dagsamIII, dagsamV}, while $\Z_{2}^{\Omega}$ suffices.  These higher-order systems are conservative over their (obvious) second-order counterparts by Section~\ref{prelim2}.\label{horka}} comprehension axioms needed to prove the latter imply second-order arithmetic by \cite{dagsamIII}*{\S3}.   We establish this and similar results in Section \ref{klipel}.  

\smallskip

In particular, we study the following theorems generalised to nets: 
the Bolzano-Weierstrass theorem (Section \ref{sha1}), the monotone convergence theorem (see Section~\ref{sha2}), the so-called anti-Specker property (Section~\ref{conets}), Cauchy nets (Section \ref{forfoxsake}), Dini's theorem (Section~\ref{sha4}), and Arzel\`a's theorem (Section \ref{shaka}).  
In each case, we shall obtain $\HBU$, and sometimes an equivalence over a reasonable base theory.  % or the weaker, but equally hard-to-prove \emph{Vitali covering theorem} from \cite{dagsamVI}.  
We also discuss \emph{unordered sums} in Section~\ref{cumsum} as the study of such sums by Moore in \cite{moorelimit1} was a step towards the Moore-Smith theory in \cite{moorsmidje}.  
%Finally, we show in Section~\ref{karmichael} that basic theorems about nets can yield the much stronger $\Pi_{1}^{1}$-comprehension when combined with arithmetical comprehension.  
%Similarly, such theorems yield $\Pi_{k+1}^{1}$-comprehension when combined with (higher-order) $\Pi_{k}^{1}$-comprehension. 
%Hence, by way of an improvement over the results in \cites{dagsamIII, dagsamV, dagsamVI}, we obtain a hierarchy parallel to the medium range of the \emph{G\"odel hierarchy}, as discussed in Section \ref{kodel}.

\smallskip

Secondly, we study the role of the Axiom of Choice.  In particular, we show that:
\begin{enumerate}
 \renewcommand{\theenumi}{\roman{enumi}}
\item replacing nets by sequences requires the Axiom of (countable) Choice, \label{goa1}
\item replacing sequences by nets can obviate the need for the latter axiom.\label{goa2}
\end{enumerate}
As to goal \eqref{goa1}, the minimal comprehension axioms needed to prove basic results about nets are rather strong, i.e.\ these minimal axioms imply full second-order arithmetic.  It may therefore seem desirable (and in line with the coding practice of classical/second-order RM) to replace the limit process involving nets by a `countable' limit process involving \emph{sequences}, i.e.\ if a net converges to some limit, then there should be a \emph{sequence} in the net that also converges to the same limit.        
This `sub-sequence property' was studied by Bourbaki (\cite{gentoporg2}) and we show in Section~\ref{NAP} that a highly elementary instance implies the \emph{Lindel\"of lemma} for $\R$, which is at least\footnote{Note that $\LIN+\WKL$ implies $\HBU$ by \cite{simpson2}*{IV.1}, and there are versions of $\LIN$ that imply fragments of the Axiom of (countable) Choice (see \cite{dagsamV}*{\S5}), in contrast to $\HBU$.} as hard to prove as $\HBU$.  
An even weaker instance is shown to be equivalent to a fragment of the Axiom of (countable) Choice, not provable in $\textsf{ZF}$.  % set theory.  
%In turn, the latter is derived from the \emph{ascending-descending sequence principle} (called $\ADS$; see e.g.\ \cite{dsliceke}) generalised to linear orders on Baire space.

\smallskip

Secondly, as to goal \eqref{goa2}, we establish in Section \ref{NAP2} the local equivalence between `epsilon-delta' continuity and the notion of continuity provided by nets \emph{without} using the Axiom of Choice; the latter axiom is essential for the equivalence involving \emph{sequential} continuity.
We prove a similar result for closed\footnote{As discussed in Section \ref{neclo}, `(sequentially) closed' sets are represented by $\R\di \R$-functions.} and sequentially closed sets in Section~\ref{neclo}.  In other words, while basic properties of nets are hard to prove, nets can also obviate the need for the Axiom of Choice, a foundationally important observation, as discussed in Section \ref{nintro}. 
Finally, we stress that the definition of closed sets in \cite{gieren} and the definition of continuity in \cites{degou, gieren} are given \emph{in terms of nets}, i.e.\ nets are central to domain theory and are used to define basic notions.  % of closed set and continuous function.  
It should be noted that the notion of \emph{Scott continuity} (also defined via nets) is more central than the aforementioned continuity notions in domain theory.  

\smallskip

Thirdly, as noted above, the main part of this paper is restricted to nets indexed by subsets of Baire space (as in Definition \ref{strijker}), as such nets suffice to obtain our main results. 
We shall study (more) general index sets in Appendix~\ref{naarhetgasthuis}.  
In particular, we obtain full $n$-th order arithmetic from a realiser (aka witnessing functional) for the monotone convergence theorem 
for nets indexed by sets expressible in the language of $n$-th order arithmetic.  Appendix \ref{naarhetgasthuis} is meant as illustration:
we believe that this kind of study should be further developed in a set theoretic framework.  
Nonetheless, index sets beyond Baire space do occur `in the wild', namely in \emph{fuzzy mathematics} and the \emph{iterated limit theorems}, as discussed in Remark \ref{fuzzytop}. 
%Thus, the Bolzano-Weierstrass theorem for nets seems incredibly hard to prove.  

\smallskip

Finally, some initial RM-results on nets, in particular certain theorems from Sections~\ref{sha1}, \ref{sha2}, \ref{sha4}, and Sections~\ref{NAP}, \ref{NAP2}, and \ref{CTH} can be found in \cite{samcie19,samwollic19} as part of LNCS conference proceedings.  
All other results in this paper are new, while the below proofs are the most elementary to date.  It goes without saying that this paper constitutes a spin-off from the joint project with Dag Normann
on the Reverse Mathematics and computability theory of the uncountable.  The interested reader may consult \cite{dagsamIII} for an introduction to this endeavour. 

\subsection{Motivation}\label{motiv}
We provide some motivation for the RM-study of nets in this section.  
In light of the previous section, the answer to the question in item \eqref{formiko} is positive: the Bolzano-Weierstrass theorem \emph{for nets} implies both sequential and (uncountable) open-cover compactness.  
\begin{enumerate}
 \renewcommand{\theenumi}{\alph{enumi}}
\item Nets were introduced\footnote{On a historical note, Vietoris introduces the notion of \emph{oriented set} in \cite{kliet}*{p.~184}, which is exactly the notion of `directed set'.  He proceeds to prove (among others)
a version of the Bolzano-Weierstrass theorem, and also mentions that these results are part of his dissertation, written in the period 1913-1919, i.e.\ during his army service for the Great War.} about a century ago (\cites{moorelimit2,moorsmidje,kliet }) and many basic theorems have since been generalised to nets, i.e.\ nets should count as `ordinary mathematics' in Simpson's sense, as discussed in \cite{simpson2}*{I.1}.
\item Nets provide an elegant equivalent formulation of compactness; the latter has been studied in remarkable detail in RM (see e.g.\ \cite{browner, brownphd}).   This paper can therefore be viewed as a continuation of this study, based on nets.  
\item Filters are studied in the RM of topology (see e.g.\ \cite{mummyphd, mummymf, mummy}), and it is well-known that nets and filters provide an equivalent framework (see \cite{zonderfilter}).  
\item Sequential compactness and open-cover compactness are classified in quite different$^{\ref{horka}}$ RM categories.  It is a natural, if somewhat outlandish, question if there is one concept that `unifies' these different notions of compactness.\label{formiko}
\item The weak-$*$-topology, including the Banach-Alaoglu theorem, is studied in RM (see \cite{simpson2}*{X.2} for an overview) and this topology has an elegant formulation in terms of nets.  Moreover, Alaoglu makes use of nets in \cite{alavielvanzijnklapstoel} to prove the general version of the aforementioned theorem.       
\item Domain theory and associated topologies are studied in RM (\cites{momg, christustepaard}), and nets take central stage in domain theory in \cite{gieren, gieren2,degou}.   
\item Nets are used in topological dynamics (\cite{furweiss}), which is studied in the proof mining program (\cite{gerwei,kohlenbach3}).  
More generally, ergodic theory is also studied in RM (\cite{day}) and proof theory (\cite{aviergo}), and it is therefore a natural question how strong basic results regarding nets are.  
\item  In general, sequences do not suffice for describing topologies, and nets are needed instead (see the \emph{Arens-Fort space} in \cite{steengoed}*{p.\ 54}). 
As it turns out, even for basic spaces like $\R$ where sequences do suffice to describe the topology \emph{over strong systems} like $\ZFC$,  
sequences no longer suffice to describe the topology \emph{over weak systems} like $\RCAo$, but nets do suffice (see Section~\ref{espero}).\label{devil} 
%is a basic space in which sequences (resp.\ nets) cannot (resp.\ can) capture the topology.  As it turns out, even ZFC vs $\RCAo$  
%$X=\N\times \N$ with a topology $\tau$ in which $(0,0)$ is a cluster point for some sequence in $X\setminus (0,0)$, but no {sequence} in $X\setminus(0,0)$ converges to $(0,0)$.  
%Thus, sequences do not suffices to define basic topological notions.
  % as also witnessed by Remark \ref{refer}.
%\item Devil's advocate: things must be crammed in soa!
\end{enumerate}
We discuss these and related conceptual motivations in more detail in the body of the paper, namely in Remarks \ref{fvn}, \ref{floppy}, \ref{weakling}, and \ref{fuzzytop}.  
We stress that item~\eqref{devil} elevates the RM-study of nets beyond that of a mere curiosity: nets are in fact \emph{needed} to described topologies in weak systems like $\RCAo$, even if the topology can be described by sequences \emph{assuming strong systems} like $\ZFC$.  In fact, countable choice plays an essential role, as discussed in Section \ref{nintro}.  
%To make a somewhat controversial claim: studying theorems in weak systems requires the right choice of definitions, and nets are the right choice.  
\section{Preliminaries}\label{prelim}
We introduce \emph{Reverse Mathematics} in Section \ref{prelim1}, as well as its generalisation to \emph{higher-order arithmetic}, and the associated base theory $\RCAo$.  % in Section \ref{KOH}.  
We introduce some essential axioms in Section~\ref{prelim2}.  
We provide a brief introduction to nets and related concepts in Section \ref{intronet}. 
%As detailed in the latter, we only study nets indexed by subsets of Baire space, i.e.\ part of third-order arithmetic, as such nets already suffice to obtain our main results.
As noted in Section \ref{schintro}, to obtain our main results it suffices to study nets \emph{indexed by subsets of Baire space}, i.e.\ part of third-order arithmetic; the associated bit of set theory shall be represented in $\RCAo$ as in Definition \ref{strijker}.  % in Section \ref{prelim1}.  

% are already general enough to obtain our main results.  
\subsection{Reverse Mathematics}\label{prelim1}
Reverse Mathematics is a program in the foundations of mathematics initiated around 1975 by Friedman (\cites{fried,fried2}) and developed extensively by Simpson (\cite{simpson2}).  
The aim of RM is to identify the minimal axioms needed to prove theorems of ordinary, i.e.\ non-set theoretical, mathematics. 

\smallskip

We refer to \cite{stillebron} for a basic introduction to RM and to \cite{simpson2, simpson1} for an overview of RM.  We expect basic familiarity with RM, but do sketch some aspects of Kohlenbach's \emph{higher-order} RM (\cite{kohlenbach2}) essential to this paper, including the base theory $\RCAo$ (Definition \ref{kase}).  
As will become clear, the latter is officially a type theory but can accommodate (enough) set theory via Definition \ref{strijker}. 

\smallskip

First of all, in contrast to `classical' RM based on \emph{second-order arithmetic} $\Z_{2}$, higher-order RM uses $\L_{\omega}$, the richer language of \emph{higher-order arithmetic}.  
Indeed, while the latter is restricted to natural numbers and sets of natural numbers, higher-order arithmetic can accommodate sets of sets of natural numbers, sets of sets of sets of natural numbers, et cetera.  
To formalise this idea, we introduce the collection of \emph{all finite types} $\mathbf{T}$, defined by the two clauses:
\begin{center}
(i) $0\in \mathbf{T}$   and   (ii)  If $\sigma, \tau\in \mathbf{T}$ then $( \sigma \di \tau) \in \mathbf{T}$,
\end{center}
where $0$ is the type of natural numbers, and $\sigma\di \tau$ is the type of mappings from objects of type $\sigma$ to objects of type $\tau$.
In this way, $1\equiv 0\di 0$ is the type of functions from numbers to numbers, and where  $n+1\equiv n\di 0$.  Viewing sets as given by characteristic functions, we note that $\Z_{2}$ only includes objects of type $0$ and $1$.    

\smallskip

Secondly, the language $\L_{\omega}$ includes variables $x^{\rho}, y^{\rho}, z^{\rho},\dots$ of any finite type $\rho\in \mathbf{T}$.  Types may be omitted when they can be inferred from context.  
The constants of $\L_{\omega}$ include the type $0$ objects $0, 1$ and $ <_{0}, +_{0}, \times_{0},=_{0}$  which are intended to have their usual meaning as operations on $\N$.
Equality at higher types is defined in terms of `$=_{0}$' as follows: for any objects $x^{\tau}, y^{\tau}$, we have
\be\label{aparth}
[x=_{\tau}y] \equiv (\forall z_{1}^{\tau_{1}}\dots z_{k}^{\tau_{k}})[xz_{1}\dots z_{k}=_{0}yz_{1}\dots z_{k}],
\ee
if the type $\tau$ is composed as $\tau\equiv(\tau_{1}\di \dots\di \tau_{k}\di 0)$.  
Furthermore, $\L_{\omega}$ also includes the \emph{recursor constant} $\mathbf{R}_{0}$, which allows for iteration on type $0$-objects as in \eqref{special}.  Formulas and terms are defined as usual.  
One obtains the sub-language $\L_{n+2}$ by restricting the above type formation rule to produce only type $n+1$ objects (and related types of similar complexity).        
\bdefi\label{kase} 
The base theory $\RCAo$ consists of the following axioms.
\begin{enumerate}
 \renewcommand{\theenumi}{\alph{enumi}}
\item  Basic axioms expressing that $0, 1, <_{0}, +_{0}, \times_{0}$ form an ordered semi-ring with equality $=_{0}$.
\item Basic axioms defining the well-known $\Pi$ and $\Sigma$ combinators (aka $K$ and $S$ in \cite{avi2}), which allow for the definition of \emph{$\lambda$-abstraction}. 
\item The defining axiom of the recursor constant $\mathbf{R}_{0}$: For $m^{0}$ and $f^{1}$: 
\be\label{special}
\mathbf{R}_{0}(f, m, 0):= m \textup{ and } \mathbf{R}_{0}(f, m, n+1):= f(n, \mathbf{R}_{0}(f, m, n)).
\ee
\item The \emph{axiom of extensionality}: for all $\rho, \tau\in \mathbf{T}$, we have:
\be\label{EXT}\tag{$\textsf{\textup{E}}_{\rho, \tau}$}  
(\forall  x^{\rho},y^{\rho}, \varphi^{\rho\di \tau}) \big[x=_{\rho} y \di \varphi(x)=_{\tau}\varphi(y)   \big].
\ee 
\item The induction axiom for quantifier-free\footnote{To be absolutely clear, variables (of any finite type) are allowed in quantifier-free formulas of the language $\L_{\omega}$: only quantifiers are banned.} formulas of $\L_{\omega}$.
\item $\QFAC^{1,0}$: The quantifier-free Axiom of Choice as in Definition \ref{QFAC}.
\end{enumerate}
\edefi
\bdefi\label{QFAC} The axiom $\QFAC$ consists of the following for all $\sigma, \tau \in \textbf{T}$:
\be\tag{$\QFAC^{\sigma,\tau}$}
(\forall x^{\sigma})(\exists y^{\tau})A(x, y)\di (\exists Y^{\sigma\di \tau})(\forall x^{\sigma})A(x, Y(x)),
\ee
for any quantifier-free formula $A$ in the language of $\L_{\omega}$.
\edefi
We let $\IND$ be the induction axiom for all formulas in $\L_{\omega}$.  The system $\RCAo+\IND$ has the same first-order strength as Peano arithmetic.  

\smallskip

As discussed in \cite{kohlenbach2}*{\S2}, $\RCAo$ and $\RCA_{0}$ prove the same sentences `up to language' as the latter is set-based and the former function-based.  Recursion as in \eqref{special} is also called \emph{primitive recursion}.  % the class of functionals obtained from $\mathbf{R}_{\rho}$ for all $\rho \in \mathbf{T}$ is called \emph{G\"odel's system $T$} of all (higher-order) primitive recursive functionals.  

\smallskip

We use the usual notations for natural, rational, and real numbers, and the associated functions, as introduced in \cite{kohlenbach2}*{p.\ 288-289}.  
\begin{defi}[Real numbers and related notions in $\RCAo$]\label{keepintireal}\rm~
\begin{enumerate}
 \renewcommand{\theenumi}{\alph{enumi}}
\item Natural numbers correspond to type zero objects, and we use `$n^{0}$' and `$n\in \N$' interchangeably.  Rational numbers are defined as signed quotients of natural numbers, and `$q\in \Q$' and `$<_{\Q}$' have their usual meaning.    
\item Real numbers are coded by fast-converging Cauchy sequences $q_{(\cdot)}:\N\di \Q$, i.e.\  such that $(\forall n^{0}, i^{0})(|q_{n}-q_{n+i}|<_{\Q} \frac{1}{2^{n}})$.  
We use Kohlenbach's `hat function' from \cite{kohlenbach2}*{p.\ 289} to guarantee that every $q^{1}$ defines a real number.  
\item We write `$x\in \R$' to express that $x^{1}:=(q^{1}_{(\cdot)})$ represents a real as in the previous item and write $[x](k):=q_{k}$ for the $k$-th approximation of $x$.    
\item Two reals $x, y$ represented by $q_{(\cdot)}$ and $r_{(\cdot)}$ are \emph{equal}, denoted $x=_{\R}y$, if $(\forall n^{0})(|q_{n}-r_{n}|\leq {2^{-n+1}})$. Inequality `$<_{\R}$' is defined similarly.  
We sometimes omit the subscript `$\R$' if it is clear from context.           
\item Functions $F:\R\di \R$ are represented by $\Phi^{1\di 1}$ mapping equal reals to equal reals, i.e. $(\forall x , y\in \R)(x=_{\R}y\di \Phi(x)=_{\R}\Phi(y))$.\label{EXTEN}
\item The relation `$x\leq_{\tau}y$' is defined as in \eqref{aparth} but with `$\leq_{0}$' instead of `$=_{0}$'.  Binary sequences are denoted `$f^{1}, g^{1}\leq_{1}1$', but also `$f,g\in C$' or `$f, g\in 2^{\N}$'.  Elements of Baire space are given by $f^{1}, g^{1}$, but also denoted `$f, g\in \N^{\N}$'.
\item For a binary sequence $f^{1}$, the associated real in $[0,1]$ is $\r(f):=\sum_{n=0}^{\infty}\frac{f(n)}{2^{n+1}}$.\label{detrippe}
\item Sets of type $\rho$ objects $X^{\rho\di 0}, Y^{\rho\di 0}, \dots$ are given by their characteristic functions $F^{\rho\di 0}_{X}\leq_{\rho\di 0}1$, i.e.\ we write `$x\in X$' for $ F_{X}(x)=_{0}1$. \label{koer} 
% where $F_{X}^{\rho\di 0}\leq_{\rho\di 0}1$.\label{koer}
\end{enumerate}
\end{defi}
The following special case of item \eqref{koer} is singled out, as it will be used frequently.
\bdefi[$\RCAo$]\label{strijker}
A `subset $D$ of $\N^{\N}$' is given by its characteristic function $F_{D}^{2}\leq_{2}1$, i.e.\ we write `$f\in D$' for $ F_{D}(f)=1$ for any $f\in \N^{\N}$.
A `binary relation $\preceq$ on a subset $D$ of $\N^{\N}$' is given by the associated characteristic function $G_{\preceq}^{(1\times 1)\di 0}$, i.e.\ we write `$f\preceq g$' for $G_{\preceq}(f, g)=1$ and any $f, g\in D$.
Assuming extensionality on the reals as in item \eqref{EXTEN}, we obtain characteristic functions that represent subsets of $\R$ and relations thereon.  
Using pairing functions, it is clear we can also represent sets of finite sequences (of reals), and relations thereon.  
\edefi
Finally, we mention the highly useful $\ECF$-interpretation.  
\begin{rem}[The $\ECF$-interpretation]\label{ECF}\rm
The technical definition of $\ECF$ may be found in \cite{troelstra1}*{p.\ 138, \S2.6}.
Intuitively speaking, the $\ECF$-interpretation $[A]_{\ECF}$ of a formula $A\in \L_{\omega}$ is just $A$ with all variables 
of type two and higher replaced by countable representations of continuous functionals.  Such representations are also (equivalently) called `associates' or `codes' (see \cite{kohlenbach4}*{\S4}). 
The $\ECF$-interpretation connects $\RCAo$ and $\RCA_{0}$ (see \cite{kohlenbach2}*{Prop.\ 3.1}) in that if $\RCAo$ proves $A$, then $\RCA_{0}$ proves $[A]_{\ECF}$, again `up to language', as $\RCA_{0}$ is 
formulated using sets, and $[A]_{\ECF}$ is formulated using types, namely only using type zero and one objects.  
\end{rem}
For completeness, we list the following notational convention on finite sequences.  
\begin{nota}[Finite sequences]\label{skim}\rm
We assume a dedicated type for `finite sequences of objects of type $\rho$', namely $\rho^{*}$.  Since the usual coding of pairs of numbers goes through in $\RCAo$, we shall not always distinguish between $0$ and $0^{*}$. 
Similarly, we do not always distinguish between `$s^{\rho}$' and `$\langle s^{\rho}\rangle$', where the former is `the object $s$ of type $\rho$', and the latter is `the sequence of type $\rho^{*}$ with only element $s^{\rho}$'.  The empty sequence for the type $\rho^{*}$ is denoted by `$\langle \rangle_{\rho}$', usually with the typing omitted.  

\smallskip

Furthermore, we denote by `$|s|=n$' the length of the finite sequence $s^{\rho^{*}}=\langle s_{0}^{\rho},s_{1}^{\rho},\dots,s_{n-1}^{\rho}\rangle$, where $|\langle\rangle|=0$, i.e.\ the empty sequence has length zero.
For sequences $s^{\rho^{*}}, t^{\rho^{*}}$, we denote by `$s*t$' the concatenation of $s$ and $t$, i.e.\ $(s*t)(i)=s(i)$ for $i<|s|$ and $(s*t)(j)=t(|s|-j)$ for $|s|\leq j< |s|+|t|$. For a sequence $s^{\rho^{*}}$, we define $\overline{s}N:=\langle s(0), s(1), \dots,  s(N-1)\rangle $ for $N^{0}<|s|$.  
For a sequence $\alpha^{0\di \rho}$, we also write $\overline{\alpha}N=\langle \alpha(0), \alpha(1),\dots, \alpha(N-1)\rangle$ for \emph{any} $N^{0}$.  By way of shorthand, 
$(\forall q^{\rho}\in Q^{\rho^{*}})A(q)$ abbreviates $(\forall i^{0}<|Q|)A(Q(i))$, which is (equivalent to) quantifier-free if $A$ is.   
\end{nota}

\subsection{Some axioms of higher-order RM}\label{prelim2}
We introduce some functionals which constitute the counterparts of second-order arithmetic $\Z_{2}$, and some of the Big Five systems, in higher-order RM.
We use the formulation from \cite{kohlenbach2, dagsamIII}.  

\smallskip
\noindent
First of all, Feferman's \emph{search operator} $\mu^{2}$ (\cite{avi2}) is defined as follows: 
\begin{align}\label{mu}\tag{$\mu^{2}$}
(\exists \mu^{2})(\forall f^{1})\big[ (\exists n)(f(n)=0) \di [f(\mu(f))=0&\wedge (\forall i<\mu(f))(f(i)\ne 0) ]\\
& \wedge  [(\forall n)(f(n)\ne0)\di   \mu(f)=0]    \big]. \notag
\end{align}
The system $\ACA_{0}^{\omega}\equiv\RCAo+(\mu^{2})$ proves the same sentences as $\ACA_{0}$ by \cite{hunterphd}*{Theorem~2.5}.   The (unique) functional $\mu^{2}$ in $(\mu^{2})$ is also called \emph{Feferman's $\mu$} for short, 
and is \emph{discontinuous} at $f=_{1}11\dots$; in fact, $(\mu^{2})$ is equivalent to the existence of $F:\R\di\R$ such that $F(x)=1$ if $x>_{\R}0$, and $0$ otherwise (\cite{kohlenbach2}*{\S3}), and to 
\be\label{muk}\tag{$\exists^{2}$}
(\exists \varphi^{2}\leq_{2}1)(\forall f^{1})\big[(\exists n)(f(n)=0) \asa \varphi(f)=0    \big]. 
\ee
\noindent
Secondly, \emph{the Suslin functional} $\SS^{2}$ is defined as follows:
\be\tag{$\SS^{2}$}
(\exists\SS^{2}\leq_{2}1)(\forall f^{1})\big[  (\exists g^{1})(\forall n^{0})(f(\overline{g}n)=0)\asa \SS(f)=0  \big], 
\ee
The system $\FIVE^{\omega}\equiv \RCAo+(\SS^{2})$ proves the same $\Pi_{3}^{1}$-sentences as $\FIVE$ by \cite{yamayamaharehare}*{Theorem 2.2}.  
% The (unique) functional $\SS^{2}$ in $(\SS^{2})$ is also called \emph{the Suslin functional} (\cite{kohlenbach2}).
By definition, the Suslin functional $\SS^{2}$ can decide whether a $\Sigma_{1}^{1}$-formula (as in the left-hand side of $(\SS^{2})$) is true or false.   We similarly define the functional $\SS_{k}^{2}$ which decides the truth or falsity of $\Sigma_{k}^{1}$-formulas; we also define 
the system $\SIXK$ as $\RCAo+(\SS_{k}^{2})$, where  $(\SS_{k}^{2})$ expresses that $\SS_{k}^{2}$ exists.  Note that we allow formulas with \emph{function} parameters, but \textbf{not} \emph{functionals} here.
In fact, Gandy's \emph{Superjump} (\cite{supergandy}) constitutes a way of extending $\FIVE^{\omega}$ to parameters of type two.

\smallskip

\noindent
Thirdly, full second-order arithmetic $\Z_{2}$ is readily derived from $\cup_{k}\SIXK$, or from:
\be\tag{$\exists^{3}$}
(\exists E^{3}\leq_{3}1)(\forall Y^{2})\big[  (\exists f^{1})Y(f)=0\asa E(Y)=0  \big], 
\ee
and we therefore define $\Z_{2}^{\Omega}\equiv \RCAo+(\exists^{3})$ and $\Z_{2}^\omega\equiv \cup_{k}\SIXK$, which are conservative over $\Z_{2}$ by \cite{hunterphd}*{Cor.\ 2.6}. 
Despite this close connection, $\Z_{2}^{\omega}$ and $\Z_{2}^{\Omega}$ can behave quite differently, as discussed in e.g.\ \cite{dagsamIII}*{\S2.2}.   The functional from $(\exists^{3})$ is also called `$\exists^{3}$', and we use the same convention for other functionals.  

\smallskip

Finally, the Heine-Borel theorem states the existence of a finite sub-cover for an open cover of certain spaces. 
Now, a functional $\Psi:\R\di \R^{+}$ gives rise to the \emph{canonical} cover $\cup_{x\in I} I_{x}^{\Psi}$ for $I\equiv [0,1]$, where $I_{x}^{\Psi}$ is the open interval $(x-\Psi(x), x+\Psi(x))$.  
Hence, the uncountable cover $\cup_{x\in I} I_{x}^{\Psi}$ has a finite sub-cover by the Heine-Borel theorem; in symbols:
\be\tag{$\HBU$}
(\forall \Psi:\R\di \R^{+})(\exists \langle y_{1}, \dots, y_{k}\rangle){(\forall x\in I)}(\exists i\leq k)(x\in I_{y_{i}}^{\Psi}).
\ee
Note that $\HBU$ is almost verbatim \emph{Cousin's lemma} (see \cite{cousin1}*{p.\ 22}), i.e.\ the Heine-Borel theorem restricted to canonical covers.  
The latter restriction does not make much of a big difference, as studied in \cite{sahotop}.
By \cite{dagsamIII, dagsamV}, $\Z_{2}^{\Omega}$ proves $\HBU$ but $\Z_{2}^{\omega}+\QFAC^{0,1}$ cannot, 
%Hence, the Heine-Borel theorem for uncountable covers as in $\HBU$ falls \emph{far} outside of the Big Five of RM, as noted at the end of Section \ref{RM}.  
and many basic properties of the \emph{gauge integral} (\cite{zwette, mullingitover}) are equivalent to $\HBU$.  
Although strictly speaking incorrect, we sometimes use set-theoretic notation, like reference to the cover $\cup_{x\in I} I_{x}^{\Psi}$ inside $\RCAo$,  to make proofs more understandable.  
Such reference can in principle be removed in favour of formulas of higher-order arithmetic.     
%Finally, the Heine-Borel theorem (aka \emph{Cousin's lemma} \cite{cousin1}*{p.\ 22}) states the existence of a finite sub-cover for an open cover of certain spaces. 
%Now, a functional $\Psi:\R\di \R^{+}$ gives rise to the \emph{canonical} cover $\cup_{x\in I} I_{x}^{\Psi}$ for $I\equiv [0,1]$, where $I_{x}^{\Psi}$ is the open interval $(x-\Psi(x), x+\Psi(x))$.  
%Hence, the uncountable cover $\cup_{x\in I} I_{x}^{\Psi}$ has a finite sub-cover by the Heine-Borel theorem; in symbols:
%\be\tag{$\HBU$}
%(\forall \Psi:\R\di \R^{+})(\exists  y_{1}, \dots, y_{k}\in I){(\forall x\in I)}(\exists i\leq k)(x\in I_{y_{i}}^{\Psi}).
%\ee
%By \cite{dagsamIII, dagsamV}, $\Z_{2}^{\Omega}$ proves $\HBU$ but $\Z_{2}^{\omega}+\QFAC^{0,1}$ cannot, 
%and many basic properties of the \emph{gauge integral} (\cite{zwette, mullingitover}) are equivalent to $\HBU$.  
%Although strictly speaking incorrect, we sometimes use set-theoretic notation, like reference to the cover $\cup_{x\in I} I_{x}^{\Psi}$ inside $\RCAo$,  to make proofs more understandable.  
%Such reference can in principle be removed in favour of (less intuitive) formulas of higher-order arithmetic.      
\subsection{Introducing nets}\label{intronet}
We introduce the notion of net and associated concepts.  
We first consider the following standard definition (see e.g.\ \cite{ooskelly}*{Ch.\ 2}).
\bdefi[Nets]\label{nets}
A set $D\ne \emptyset$ with a binary relation `$\preceq$' is \emph{directed} if
\begin{enumerate}
 \renewcommand{\theenumi}{\alph{enumi}}
\item The relation $\preceq$ is transitive, i.e.\ $(\forall x, y, z\in D)([x\preceq y\wedge y\preceq z] \di x\preceq z )$.
\item For $x, y \in D$, there is $z\in D$ such that $x\preceq z\wedge y\preceq z$.\label{bulk}
\item The relation $\preceq$ is reflexive, i.e.\ $(\forall x\in D)(x \preceq x)$.  
\end{enumerate}
For such $(D, \preceq)$ and topological space $X$, any mapping $x:D\di X$ is a \emph{net} in $X$.  
We denote $\lambda d. x(d)$ as `$x_{d}$' or `$x_{d}:D\di X$' to suggest the connection to sequences.  
The directed set $(D, \preceq)$ is not always explicitly mentioned together with a net $x_{d}$.
\edefi 
In this paper, we only study directed sets that are subsets of Baire space, i.e.\ as given by Definition \ref{strijker}.   
Similarly, we only study nets $x_{d}:D\di \R$ where $D$ is a subset of Baire space.  Thus, a net $x_{d}$ in $\R$ is just a type $1\di 1$ functional with extra 
structure on its domain $D$ provided by `$\preceq$' as in Definition \ref{strijker}.  
We shall allow for additional input variables over $\R$ in Section \ref{shasup2} in the study of \emph{nets of functions}.    
%To be absolutely clear, within $\RCAo$, a directed set is a pair $(D, \preceq)$ such that $D\subset \N^{\N}$ and $\preceq \subset (\N^{\N}\times \N^{\N})$ satisfying the previous 
%The relation `$\preceq$' is assumed to have domain $D$; 

%Below, we (only) study examples of $D$ with the cardinality of $\R$, i.e.\ a `step up' from sequences already gives rise to $\HBU$.  

\smallskip

The definitions of convergence and increasing net have the usual form in this setting.  
\bdefi[Convergence of nets]\label{convnet}
If $x_{d}$ is a net in $X$, we say that $x_{d}$ \emph{converges} to the limit $\lim_{d} x_{d}=y\in X$ if for every neighbourhood $U$ of $y$, there is $d_{0}\in D$ such that for all $e\succeq d_{0}$, $x_{e}\in U$. 
%Similarly, a point $y\in X$ is a \emph{limit point} of $x_{d}$ if for every neighbourhood, there is $d\in D$ such that $x_{d}\in U$.
\edefi
\bdefi[Increasing nets]
A net $x_{d}:D\di \R$ is \emph{increasing} if $a\preceq b$ implies $x_{a}\leq_{\R} x_{b} $ for all $a,b\in D$.
\edefi
\bdefi\label{clusterf}
A point $x\in X$ is a \emph{cluster point} for a net $x_{d}$ in $X$ if every neighbourhood $U$ of $x$ contains $x_{u}$ for some $u\in D$.
\edefi
The previous definition yields the following nice equivalence: a toplogical space is compact if and only if every net therein has a cluster point (\cite{zonderfilter}*{Prop.\ 3.4}).
All the below results can be formulated using cluster points \emph{only}, but such an approach does not address the question of what the counterpart of `sub-sequence' for nets is. 
Indeed, an obvious next step following Definition \ref{clusterf} is to take smaller and smaller neighbourhoods around the cluster point $x$ and (somehow) say that the associated points $x_{u}$ net-converge to $x$.   
To this end, we consider the following definition, first introduced by Moore in \cite{moringpool}, and used by Kelley in \cite{ooskelly}.  
Alternative definitions involve extra requirements (see \cite{zot}*{\S7.14}), i.e.\ our definition is the weakest. 
\bdefi[Sub-nets]\label{demisti}
A \emph{sub-net} of a net $x_{d}$ with directed set $(D, \preceq_{D})$, is a net $y_{b}$ with directed set $(B, \preceq_{B})$ such that there is a function $\phi : B \di D$ such that:
\begin{enumerate}
 \renewcommand{\theenumi}{\alph{enumi}}
\item the function $\phi$ satisfies $ y_{b} = x_{\phi(b)},$
\item $(\forall d\in D)(\exists b_{0}\in B)(\forall b\succeq_{B} b_{0})(\phi(b)\succeq_{D} d)$.
%?d?D,?b0 ?B such that if b\UTF{227B}b0 then ?(b)\UTF{227B}d.
\end{enumerate}
\edefi
We point out that the distinction between `$\preceq_{B}$' and `$\preceq_{D}$' is not always made in the literature (see e.g.\ \cite{zonderfilter, ooskelly}).  
Finally, $\N$ with its usual ordering yields a directed set, i.e.\ convergence results about 
nets do apply to sequences.  Of course, a \emph{sub-net} of a sequence is not necessarily a \emph{sub-sequence}, i.e.\ some care is advisable in these matters.  
Nonetheless, the Bolzano-Weierstrass theorem \emph{for nets} will be seen to imply the {monotone convergence theorem} \emph{for sequences} in Section~\ref{sha1}.

\section{Main results I}\label{klipel}
We study the generalisation to nets of theorems pertaining to the \emph{sequential compactness} of the unit interval in Section \ref{shasup1}.
We study theorems pertaining to \emph{nets of functions} in Section \ref{shasup2}.  
In each case, we obtain $\HBU$ from Section \ref{prelim2}, and sometimes even an equivalence over a reasonable base theory.   %or the weaker, but equally hard-to-prove, \emph{Vitali covering theorem} from \cite{dagsamVI}.
%In Section~\ref{karmichael}, we show that the first-order strength of basic theorems about nets can also `explode', i.e.\ increase dramatically when combined with e.g.\ $(\exists^{2})$ or $(\SS_{k}^{2})$.

\subsection{Sequential compactness and nets}\label{shasup1}
In this section, we study the generalisation to nets of theorems pertaining to the sequential compactness of the unit interval, like the Bolzano-Weierstrass (Section \ref{sha1}) and the monotone convergence theorem (Section \ref{sha2}).
These generalisations imply the sequential compactness of the unit interval, but also the Heine-Borel compactness \emph{for uncountable covers} as in $\HBU$.  Hence, nets provide a `unified' approach to compactness that captures both sequential and (uncountable) open-cover compactness.   
We also study the `anti-Specker property' from constructive mathematics in Section \ref{conets}, which can be (equivalently) viewed as the study of isolated points of nets.  
Basic theorems pertaining to Cauchy nets are studied in Section \ref{forfoxsake}.
Finally, we discuss \emph{unordered sums} in Section~\ref{cumsum} as the study of such sums by Moore in \cite{moorelimit1} was the first step towards the Moore-Smith theory of convergence in \cite{moorsmidje}.

\subsubsection{The Bolzano-Weierstrass theorem for nets}\label{sha1}
We study the Bolzano-Weierstrass theorem for nets, $\BW_{\net}$ for short, i.e.\ the statement that a net in the unit interval $I\equiv [0,1]$ has a convergent sub-net.  
This theorem is one of the standard results pertaining to nets, and can even be found in mathematical physics, namely in \cite{reedafvegen}*{p.\ 98}.  
As discussed in Section \ref{intronet}, $\BW_{\net}$ is limited to nets indexed by subsets of $\N^{\N}$.  
%In this paper, we only study directed sets that are subsets of Baire space, i.e.\ as given by Definition \ref{strijker}.   
\begin{thm}\label{merdes}
The system $\RCAo+\BW_{\net}$ proves $\HBU$.
\end{thm}
\begin{proof}
Note that $\BW_{\net}$ implies the monotone convergence theorem \emph{for sequences}, as the latter are nets.  
Indeed, if a sub-net $x_{\phi(b)}$ of an increasing sequence $x_{n}$ converges to $x=\lim_{b}x_{\phi(b)}$, then also $\lim_{n\di \infty}x_{n}=x$. 
Hence, we have access to $\ACA_{0}$ by \cite{simpson2}*{III.2.2}.  
Now, in case $\neg(\exists^{2})$, all functions on $\R$ are continuous by \cite{kohlenbach2}*{Prop.~3.12}, 
and $\HBU$ reduces to $\WKL$ by \cite{kohlenbach4}*{\S4}.
We now prove $\HBU$ in case $(\exists^{2})$, which finishes the proof using the law of excluded middle.
Thus, suppose $\neg\HBU$ and fix some $\Psi:I\di \R^{+}$ for which $\cup_{x\in I}I_{x}^{\Psi}$ does not have a finite sub-cover.
Let $D$ be the set of all finite sequences of reals in the unit interval, and define `$v \preceq_{D} w$' for $w, v\in D$ if $\cup_{i<|v|}I_{v(i)}^{\psi}\subseteq \cup_{i<|w|}I_{w(i)}^{\psi}$, i.e.\ 
the cover generated by $w$ includes the cover associated to $v$.  Note that $(\exists^{2})$ suffices to define $\preceq_{D}$.  Clearly, the latter is transitive and reflexive, and item \eqref{bulk} in Definition~\ref{nets} is satisfied by noting that $(v*w) \succeq w$ and $(v*w) \succeq v$.  To define a net, consider 
\be\label{okl}
(\forall w^{1^{*}}\in D)(\exists q \in \Q\cap [0,1])\underline{(q\not \in \cup_{i<|w|}I_{w(i)}^{\Psi})}, 
\ee
which again holds by assumption.  Note that the underlined formula in \eqref{okl} is decidable thanks to $(\exists^{2})$.  
Applying $\QFAC^{1,0}$ to \eqref{okl}, we obtain a net $x_{w}$ in $[0,1]$, which has a convergent (say to $z_{0}\in I$) sub-net $y_{b}=x_{\phi(d)}$ for some directed set $(B, \preceq_{B})$ and $\phi:B\di D$, by $\BW_{\net}$.
By definition, the neighbourhood $U_{0}=I_{z_{0}}^{\Psi}$ contains all $y_{b}$ for $b\succeq_{B} b_{1}$ for some $b_{1}\in B$.  However, taking $d=\langle z_{0} \rangle\in D$ in the second item in Definition \ref{demisti}, there is also $b_{0}\in B$ such that $(\forall b\succeq_{B} b_{0})(\phi(b)\succeq_{D} \langle y_{0}\rangle)$.
By the definition of `$\preceq_{D}$', $\phi(b)$ is hence such that $\cup_{i<|\phi(b)|}I_{\phi(b)(i)}^{\Psi}$ contains $U_{0}$, for any $b\succeq_{B} b_{0}$.  Now use item~\eqref{bulk} from Definition~\ref{nets} (for the directed set $(B,\preceq_{B})$) to find $b_{2}\in B$ satisfying $b_{2}\succeq_{B} b_{0}$ and $b_{2}\succeq_{B} b_{1}$.  
Hence, $y_{b_{2}}=x_{\phi(b_{2})}$ is in $U_{0}$, but $\cup_{i<|\phi(b_{2})|} I^{\Psi}_{\phi(b_{2})(i)}$ also contains $U_{0}$, i.e.\ $x_{\phi(b_{2})}$ must be \emph{outside} of $U_{0}$ by the definition of $x_{w}$, a contradiction. 
In this way, we also obtain $\HBU$ in case $(\exists^{2})$.  % and we are done.
\end{proof}
We cannot expect a reversal in the previous theorem, as $\BW_{\net}$ implies $\ACA_{0}$, while $\RCAo+\HBU$ is conservative over $\WKL_{0}$, which readily follows from applying the $\ECF$-translation from Remark \ref{ECF}.  
Furthermore, the theorem suggests a realiser (aka witnessing functional) for $\BW_{\net}$ would compute a realiser for the Bolzano-Weierstrass theorem, and hence $\exists^{2}$, as well as a realiser for $\HBU$, called $\Theta$ in \cites{dagsam, dagsamII, dagsamIII}.  By the results in the latter, a realiser for $\BW_{\net}$ therefore would compute a realiser for $\ATR_{0}$.  %In the next section, we discuss why the theorem does not depend on the definition of sub-net.  
We obtain \emph{much} stronger results in Section \ref{sha2}.  % and \ref{kawl}. 

\subsubsection{The monotone convergence theorems for nets}\label{sha2}
We study the monotone convergence theorem for nets in the unit interval.
To this end, let $\MCT_{\net}^{0}$ state that every increasing net in the unit interval converges.  
As discussed in Section \ref{intronet}, $\MCT_{\net}^{0}$ is restricted to nets that are indexed by subsets of Baire space.  

\smallskip

We show $\MCT_{\net}^{0}\di\HBU$ in Theorem \ref{merdescor}, but Corollary \ref{kraft} is of more importance, as it establishes that $\MCT_{\net}^{0}$ is provable without the Axiom of Choice, i.e.\ the `hardness' of the former theorem has nothing to do with the latter.  
We obtain a relative computability result in Corollary \ref{koon}, the foundation for Appendix~\ref{naarhetgasthuis}.  We also obtain the equivalence $\MCT_{\net}^{0}\asa \HBU$ over a fairly nice base theory. 

\smallskip

 As to the provenance of $\MCT_{\net}^{0}$, this theorem can be found in e.g.\ \cite{obro}*{p.~103}, but is also implicit in domain theory (\cites{gieren, gieren2}).  
Indeed, the main objects of study of domain theory are \emph{dcpos}, i.e.\ directed-complete posets, and every monotone net converges to its supremum in any dcpo.
%As to the provenance of $\MCT_{\net}$, this theorem can be found in e.g.\ \cite{obro}*{p.\ 103}, but is also implicit in the notion of \emph{Scott domain} (and many other notions from domain theory): one of the requirements for the latter is \emph{bounded-completeness} (\cite{gieren}*{O-2.1}), i.e.\ every bounded sub-set has a supremum.  Clearly, any increasing net $x_{d}:D\di [0,1]$ converges to the supremum of the bounded set $\{x_{d}:d\in D\}$ if this real number exists.     
\begin{thm}\label{merdescor}
The system $\RCAo+\MCT_{\net}^{0}$ proves $\HBU$.
\end{thm}
\begin{proof}
We make use of $(\exists^{2})\vee \neg(\exists^{2})$ as in the proof of Theorem \ref{merdes}.  The first part involving $\neg(\exists^{2})$ is identical.  
For the second part, fix some $\Psi:I\di \R^{+}$ and use $(\exists^{2})$ to define $D$ as the set of finite sequences of reals $w^{1^{*}}$ such that $0\in w$ 
and the cover $\cup_{i<|w|}I_{w(i)}^{\Psi}$ has `no holes', i.e.\ any point between two intervals of this cover, is also in the cover.  We define `$v\preceq w$' as $(\forall i<|v|)(\exists j<|w|)(v(i)=_{\R}w(j))$.  
Clearly, $(D, \preceq)$ is a directed set and we define the net $x_{w}:D\di [0,1]$ as the right end-point of the right-most interval in $\cup_{i<|w|}I_{w(i)}^{\Psi}$, capped by $1$ if necessary.  

\smallskip

Since $x_{w}$ is increasing by definition, let $x\in [0,1]$ be the limit provided by $\MCT_{\net}^{0}$. 
If $x=_{\R}1$, then apply $\lim x_{d}= x$ for $\eps=\Psi(1)$ to find a finite sub-cover for the canonical cover associated to $\Psi$.
In case $x<_{\R}1$, apply $\lim_{d} x_{d}= x$ for $\eps_{0}=\min(\Psi(x), |x-1|/2)$, i.e.\ there is $w_{0}\in D$ such that for all $v\succeq w_{0}$, we have $|x_{v}-x|<\eps_{0}$, implying $x_{v}\in I_{x}^{\Psi}$.
Fix such $ w_{0}$ and consider $v_{0}:=w_{0}*\langle x\rangle$.  The latter is in $D$ and satisfies $v_{0}\succeq w_{0}$.  Hence, $x_{v_{0}}$ must be in $I_{x}^{\Psi}$ by the aforementioned convergence, but $x_{v_{0}}\not \in I_{x}^{\Psi}$ by definition of the net $x_{w}$.  
Hence, we have obtained a contradiction in case $x<_{\R}1$, and we are done.    
\end{proof}
The previous proof is counter-intuitive as it does not go through for a \emph{sequence} defined as $x_{0}:=0$ and $x_{n+1}:= x_{n}+ \Psi(x_{n})$.
In fact, Borel's original proof of the Heine-Borel theorem (\cite{opborrelen}) is based on transfinite iteration of these kinds of sequences. 
Of course, one could define $x_{n+1}$ as the right end-point of the `largest' interval covering $x_{n}$, but this requires $\exists^{3}$.  In this light, the proof of the theorem involving nets is `more constructive' than a proof involving sequences and $\exists^{3}$.

\smallskip

On one hand, the previous implies that nets indexed by subsets of Baire space already give rise to $\HBU$.
On the other hand, the proof of the following corollary suggests that such nets are `all we can handle' in $\Z_{2}^{\Omega}$.  
%Moreover, the proof suggests that $\Z_{2}^{\omega}+\HBU$ cannot prove $\MCT_{\net}^{0}$.
\begin{cor}\label{kraft}
The system $\Z_{2}^{\Omega}$ proves $\MCT^{0}_{\net}$, while $\Z_{2}^{\omega}+\QFAC^{0,1}$ does not.  
\end{cor}
\begin{proof}
The negative result follows from \cite{dagsamV}*{Theorem 4.3}.  For the remaining result, note that $\HBU$ is available thanks to \cite{dagsamV}*{Theorem 4.2}.
Suppose $\neg\MCT_{\net}^{0}$, i.e.\ there is some increasing net $x_{d}$ in $I$ that does not converge to any point in $I$.  Hence, for every $x\in I$ there is $n\in \N$ such that for all $d\in D$ there is $e\succeq d$ such that $|x-x_{e}|\geq \frac{1}{2^{n}}$.  
Since $\exists^{3}$ is given, we may use $\QFAC^{1,0}$ to obtain $\Phi:I\di \R$ such that $\Phi(x)$ is the least such $n\in \N$.  Define $\Psi(x):=\frac{1}{2^{\Phi(x)}}$ and use $\HBU$ to find $y_{1}, \dots, y_{k}\in I$ such that $\cup_{i\leq _{k}}I_{y_{i}}^{\Psi}$ covers $I$.  By definition, for any $i\leq k$, either $x_{d}$ is `below' $I_{y_{i}}^{\Psi}$ for all $d\in D$ or there is $d_{i}\in D$ such that $x_{e}$ is `above' $I_{y_{i}}^{\Psi}$ for all $e\succeq d_{i}$.  Let $d_{i_{1}}, \dots, d_{i_{m}}\in D$ be all such numbers from the second case.  There is $e_{0}\succeq d_{i_{j}}$ for $j\leq m$ by Definition \ref{nets}, but $x_{e_{0}}$ cannot be in $I$, a contradiction.         
\end{proof}
%The previous theorem also implies that $\MCT_{\net}^{0}$ has the same first-order strength as $\ACA_{0}$ using the above `excluded middle trick' and the `splitting' of $(\exists^{3})$ involving $(\exists^{2})$, as discussed in \cite{samsplit}.  
The previous theorem also implies that $\MCT_{\net}^{0}$ has the same first-order strength as $\ACA_{0}$ using the above `excluded middle trick' and the `splitting' of $(\exists^{3})$ as $[(\kappa_{0}^{3})+(\exists^{2})]\asa (\exists^{3})$, where $(\kappa_{0}^{3})$ may be found in \cite{samsplit}*{\S3.1}.  

\smallskip

Nonetheless, it remains desirable to derive $\MCT_{\net}^{0}$ from `more constructive' axioms than $(\exists^{3})$, preferably involving $\HBU$.  
To this end, recall the \emph{neighbourhood function principle} $\NFP$, a choice principle in the intersection of both classical \emph{and intuitionistic} mathematics, as discussed in \cite{troeleke1}*{p.\ 215}.    
The proof of the Lindel\"of lemma for $\R$ in $\Z_{2}^{\Omega}+\QFAC^{0,1}$ in \cite{dagsamIII} makes use of $\NFP$.
We use the following special case of $\NFP$ not involving RM codes. 
\bdefi[$\NFP_{0}$]
 For any $\Pi_{\infty}^{1}$-formula $A$ with any type two parameter:
\[
(\forall f^{1})(\exists n^{0})A(\overline{f}n)\di (\exists \Phi^{2})(\forall f^{1})A(\overline{f}\Phi(f)).
\]
\edefi
Note that $\NFP$ states the existence of an RM code for $\Phi$ as in $\NFP_{0}$.  Thus, the $\ECF$-translation of $\NFP_{0}$ yields $\NFP$ restricted to $\L_{2}$-formulas.  
Assuming $\RCAo+\NFP_{0}$ is consistent, it therefore cannot prove e.g.\ $(\exists^{2})$, while (second-order) comprehension follows by the results in \cite{troeleke1}*{p.\ 245}.
%Here, `$\gamma^{1}\in K_{0}$' expresses that $\gamma^{1}$ is an \emph{associate}, which is the same as a \emph{code} from RM by \cite{kohlenbach4}*{Prop.\ 4.4}.  
%Formally, `$\gamma^{1}\in K_{0}$' is the following formula:
%\[\textstyle
%(\forall f^{1})(\exists n^{0})(\gamma(\overline{f}n)>_{0}0) \wedge (\forall n^{0}, m^{0},f^{1}, )(m>n \wedge \gamma(\overline{f}n)>0\di  \gamma(\overline{f}n)=_{0} \gamma(\overline{f}m) ).
%\]
%The value $\gamma(f)$ for $\gamma\in K_{0}$ is defined as the unique $\gamma(\overline{f}n)-1$ for $n$ large enough.  
We now have the following theorem.  % with a fairly nice basic theory. 
\begin{thm}\label{ochtendrood}
The system $\RCAo+\IND+\NFP_{0}+\HBU$ proves $\MCT_{\net}^{0}$.
\end{thm}
\begin{proof}
Let $x_{d}:D\di I$ be an increasing net that does not converge, i.e.\ 
\be\label{kkkontro}\textstyle
(\forall y\in I)(\exists k^{0}){(\forall d\in D)(\exists e\succeq d)(|x_{e}-y|\geq \frac{1}{2^{k}})}.
\ee
Recall that $[x](k)$ is the $k$-th approximation of the real $x$; \eqref{kkkontro} implies
\be\label{kkkontro2}\textstyle
(\forall y\in I)(\exists k^{0})\underline{(\forall d\in D)(\exists e\succeq d)(|x_{e}-[y](2^{k+1})|\geq \frac{1}{2^{k}})}.
\ee
The underlined formula in \eqref{kkkontro2} can be written $A(\overline{y}k)$ with only slight abuse of notation. 
Applying $\NFP_{0}$ to \eqref{kkkontro2}, %we obtain a choice functional $\Phi^{2}$ for \eqref{kkkontro2}.
%Applying $\QFAC^{1,0}$ to the first conjunct of the latter formula,
 there is $\Phi^{2}$ such that 
 \[\textstyle
 (\forall y\in I){(\forall d\in D)(\exists e\succeq d)(|x_{e}-y |\geq \frac{1}{2^{\Phi(y)}})}.
 \]
The canonical cover $\cup_{x\in [0,1]}I_{x}^{\Psi}$ of $[0,1]$ for $\Psi$ defined as $\Psi(x):=\frac{1}{2^{\Phi(x)}}$ has a finite sub-cover $y_{0}, \dots, y_{k}$ by $\HBU$, i.e.\ $\cup_{i\leq k}I_{y_{i}}^{\Psi}$ covers $[0,1]$.
Now for $y_{i_{0}}$ such that $0\in I_{y_{i_{0}}}^{\Psi}$ and some $d_{0}\in D$, there is $e_{0}\succeq d_{0}$ such that $x_{e_{0}}\not \in I_{y_{i_{0}}}^{\Psi}$, which implies $x_{e_{0}}\geq \Psi(y_{i_{0}})$.  
Repeat the previous for $y_{i_{1}}$ such that $\Psi(y_{0})\in I_{y_{i_{1}}}^{\Psi}$ and $e_{0}$, yielding $x_{e_{1}}\geq \Psi(y_{i_{0}})+\Psi(y_{i_{1}})$ for some $e_{1}\succeq e_{0}$.  After at most $k$ steps, we find $x_{d}$ that falls outside of $[0, 1]$, a contradiction. 
Note that this $k$-step process can be performed in $\RCAo+\IND$.
\end{proof}
\begin{cor}\label{ochtendroodcor}
The system $\RCAo+\IND+\NFP_{0}$ proves $\HBU\asa \MCT_{\net}^{0}$.
\end{cor}
The axiom $\NFP_{0}$ is clearly much too strong for the above and we study weaker axioms in \cite{samph}. 
While the previous proof proceeds via contradiction, a `direct' proof is available for the case of the anti-Specker property in Section \ref{conets}.  
%with the following system `intermediate' between $\SIX$ and $\SIX^{\omega}$.
%\bdefi[$\SIX^{*}$]
%The system $\SIX^{*}$ consists of $\RCAo+\SIX$ plus
%\be\label{gedanke}
%(\forall Y^{2})(\exists X^{1})(\forall n^{0})\big[ n\in X \asa (\forall f^{1})(\exists g^{1})(Y(f, g, n)=0)    \big]. 
%\ee 
%\edefi
%The $\ECF$-translation readily\footnote{As discussed in Remark \ref{ECF}, the $\ECF$-translation replaces the functional $Y^{2}$ in \eqref{gedanke} by a total associate $\alpha^{1}$, i.e.\ the right-hand side of \eqref{gedanke} is thus $(\forall f^{1})(\exists g^{1})(\exists m^{0})(\alpha(\overline{f}m, \overline{g}m, n)=1)  $.} establishes that $\SIX^{*}$ proves the same second-order sentences as $\SIX$, but cannot prove e.g.\ $(\exists^{2})$.
%\begin{cor}
%The system $\SIX^{*}+\IND$ proves $\HBU\asa \MCT_{\net}^{0}$.
%\end{cor}
%This result is `almost optimal' because $\ACAo+\MCT_{\net}^{0}$ proves $\FIVE$ (\cite{samph}).  

\smallskip

Next, it is well-known that $\exists^{2}$ computes a realiser for the monotone convergence theorem for \emph{sequences} via a term of G\"odel's $T$, 
and vice versa (see \cite{yamayamaharehare}*{\S4}).  
Inspired by this observation, we obtain an elegant `one type up' generalisation in Corollary~\ref{koon}.        
A realiser for $\MCT_{\net}^{0}$ is a functional taking as input $(D, \preceq_{D}, x_{d})$ and outputting the real $x=\lim_{d} x_{d}$ if the inputs satisfy the conditions of $\MCT_{\net}^{0}$. 
\begin{cor}\label{koon}
A realiser for $\MCT_{\net}^{0}$ computes $\exists^{3}$ via a term of G\"odel's $T$, and vice versa.
\end{cor}
\begin{proof}
For the `vice versa' direction, one uses the usual `interval halving technique' where $\exists^{3}$ is used to decide whether there is $d\in D$ such that $x_{d}$ is in the relevant interval.  
Indeed, define $\r:C\di [0,1]$ as $\r(f):=\sum_{n=0}^{\infty} \frac{f(n)}{2^{n+1}}$ and define $f_{0}\in C$ as follows: $f_{0}(0)=1$ if and only if $(\exists d\in D)(x_{d}\geq \frac{1}{2})$ and $f_{0}(n+1)=1$ if and only if $(\exists d\in D)(x_{d}\geq \r(\overline{f_{0}}n*00\dots))$. 
Then $\lim_{d}x_{d}=\r(f_{0})$, as required. 

\smallskip

For the other direction, fix $Y^{2}$, let $D$ be Baire space, and define `$f\preceq g$' by $Y(f)\geq_{0} Y(g)$ for any $f, g\in D$.     
It is straightforward to show that $(D, \preceq)$ is a directed set.  
Define the net $x_{d}:D \di I$ by $0$ if $Y(d)>0$, and $1$ if $Y(d)=0$, which is increasing by definition.  
Hence, $x_{d}$ converges, say to $y_{0}\in I$, and if $y_{0}>_{\R}1/3$, then there must be $f^{1}$ such that $Y(f)=0$, while if $y_{0}<_{\R}2/3$, then $(\forall f^{1})(Y(f)>0)$.   
Clearly, this provides a term of G\"odel's $T$ that computes $\exists^{3}$.  
\end{proof}
The previous two corollaries show that $\MCT_{\net}^{0}$ is extremely hard to prove (in terms of the usual hierarchy of comprehension axioms), the limit therein similarly hard to compute.  
We establish in Appendix \ref{naarhetgasthuis} that generalisations of $\MCT_{\net}^{0}$ to `larger' index sets have yet more extreme properties, even compared to e.g.\ $\exists^{3}$.  %In Appendix \ref{naarhetgasthuis}, we also discuss a `sequential' version of $\MCT_{\net}^{0}$ that implies $\FIVE$ and is part of third-order arithmetic.    

\smallskip

Finally, $\BW_{\net}$ implies $\MCT_{\net}^{0}$, but the reversal seems to need the following theorem, which is restricted as in Definition \ref{strijker}; the general case is in e.g.\ \cite{andgrethel}*{\S2}.
\bdefi[$\ADS_{\net}$]
A net in $\R$ has a monotone sub-net. 
\edefi
%%The $\ECF$-interpretation of $\ADS_{\net}$ is just $\ADS$ from the RM zoo; as the latter is independent of $\WKL$ (see \cite{dsliceke}), $\ADS_{\net}$ and $\HBU$ are too.  
We conjecture $\ADS_{\net}$ does not follow from $\MCT_{\net}^{0}$ and is connected to $\ADS$ from the RM zoo (see \cite{dsliceke}).
The usual proof of $\ACA_{0}\di \ADS$ provides a proof of $[(\exists^{3})+\QFAC^{1,1}]\di \ADS_{\net}$, and we believe that the Axiom of Choice is essential.
We finish this section with a conceptual remark. 
\begin{rem}[Filters versus nets]\label{fvn}\rm
For completeness, we discuss the intimate connection between filters and nets.  Now, a topological space $X$ is compact if and only if every \emph{filter base} has a \emph{refinement} that converges to some point of $X$ (see \cite{zonderfilter}*{Prop.\ 3.4}).  
Whatever the meaning of the italicised notions, the similarity with the Bolzano-Weierstrass theorem for nets is obvious, and not a coincidence: for every net $\mathfrak{r}$, there is an associated filter base $\mathfrak{B(r)}$ such that if the erstwhile converges, so does the latter to the same point; 
one similarly associates a net $\mathfrak{r(B)}$ to a given filter base $\mathfrak{B}$ with the same convergence properties (see \cite{zonderfilter}*{\S2}).  Hence, one can reformulate $\BW_{\net}$ using filters and obtain the same result as in Theorem \ref{merdes}.  We choose nets over filters in this paper for the following reasons.
\begin{enumerate}
\item Nets have a greater intuitive clarity compared to filters, in our opinion, due to the similarity between nets and sequences. 
\item Nets are `more economical' in terms of ontology: consider the aforementioned filter base $\mathfrak{B(r)}$ associated to the net $\mathfrak{r}$.  
By \cite{zonderfilter}*{Prop.\ 2.1}, the base has strictly higher type than the net.  The same holds for $\mathfrak{r(B)}$ versus $\mathfrak{B}$.  
\item The notion of \emph{refinement} mirrors the notion of sub-net by \cite{zonderfilter}*{\S2}.  The former is studied in \cite{sahotop} in the context of paracompactness and the associated results suggest that the notion of sub-net works better in weak systems. 
\end{enumerate}
On a historical note, G.\ Birkhoff introduces what we nowadays call `convergence of a filter base' in \cite{berkhofca}, but switched to nets for \cite{berkhof}.
Despite Birkhoff's aforementioned work, Cartan is generally credited with pioneering the use of filters in topology in \cite{cartman}, and the latter are unsurprisingly also the \emph{lingua franca} of Bourbaki (\cites{gentoporg2, gentoporg1}).  On a conceptual note, the well-known notion of \emph{ultrafilter} corresponds to the equivalent notion of \emph{universal net} (\cite{zonderfilter}*{\S3}).  
\end{rem}
\subsubsection{Isolated points and nets}\label{conets}
We study a theorem pertaining to \emph{isolated points}, i.e.\ any net convergent to such a point must be eventually constant.  
Indeed, the proof of Theorem \ref{kifu} deals with $[0,1]\cup\{2\}$, which has an obvious isolated point.  There is a `constructive' dimension to this section, as discussed in Remark \ref{floppy}, where we also explain the name of the following definition. 
\bdefi[Anti-Specker property]\label{ASS}
~
\begin{enumerate}
\item We say that the net $x_{d}:D\di \R$ is \emph{eventually bounded away} from the point $x\in \R$ if $(\exists \delta>0, d\in D)(\forall e\succeq d)(|x-x_{e}|\geq \delta)$.
\item We say that the net $x_{d}:D\di \R$ is \emph{eventually bounded away} from the set $[0,1]$ if $(\exists \delta>0, d\in D)(\forall e\succeq d)(\forall x\in [0,1])(|x-x_{e}|\geq \delta)$.
\item The theorem $\AS_{\net}$ states that any net that is eventually bounded away from every $x\in [0,1]$, is eventually bounded away from $[0,1]$.
\end{enumerate}
\edefi
As discussed in Section \ref{intronet}, $\AS_{\net}$ is restricted to nets that are indexed by Baire space.  
Note that we avoid the (explicit) use of one-point extensions in our version of the anti-Specker property $\AS_{\net}$.
\begin{thm}\label{kifu}
The system $\RCAo$ proves $\AS_{\net}\di \HBU$.
\end{thm}
\begin{proof}
Since sequences are nets, it is straightforward to derive the monotone convergence theorem for sequences from $\AS_{\net}$, and hence $\ACA_{0}$ by \cite{simpson2}*{III.2.2}.  
Thus, in case $\neg(\exists^{2})$, we have $\HBU$ as the latter reduces to $\WKL$.
In case $(\exists^{2})$, define $D$ and $\preceq$ as in the proof of Theorem \ref{merdes}.  Fix $\Psi:I\di \R^{+}$ and define the net $x_{w}$ as $2$ if $\cup_{i<|w|}I_{w(i)}^{\Psi}$ covers all rationals in $[0,1]$, and otherwise use $\mu^{2}$ to find some $q\in [0,1]\cap \Q$ not in $\cup_{i<|w|}I_{w(i)}^{\Psi}$.
For any $x\in [0,1]$ and $w\succeq \langle x\rangle$, we have $|x-x_{w}|\geq \min(1,\Psi(x))$, i.e.\ $x_{w}$ is eventually bounded away from any point in $[0,1]$.  By $\AS_{\net}$, $x_{w}$ is eventually bounded away from $[0,1]$, i.e.\ there is $w_{0}\in D$ such that for $v\succeq w_{0}$ we have $x_{v}=2$.  
Clearly, this yields a finite sub-cover for the rationals in $[0,1]$, which becomes a finite sub-cover for all reals in $[0,1]$ by including in the former cover all the end-points and associated intervals.  
\end{proof}
We could weaken $\AS_{\net}$ to reflect the `limited' anti-Specker property from \cite{bridentma}; we would still obtain $\HBU$ as it suffices for the above proof that there is \emph{one} $w\in D$ such that $x_{w}=2$, which is the content of the `limited' anti-Specker property (for sequences).    
% Same for the anti-Specker property restricted to increasing nets.   
We could also derive $\MCT_{\net}^{0}$ from $\AS_{\net}$ and use Theorem \ref{merdescor}, but the previous proof is more elegant. 
An equivalence in Theorem \ref{kifu} seems difficult, in light of the type one quantifiers in the definitions of (net) convergence and related notions.  The following corollary does follow in the same way as Corollary \ref{kraft}.
\begin{cor}\label{krafter}
The system $\Z_{2}^{\Omega}$ proves $\AS_{\net}$, while $\Z_{2}^{\omega}+\QFAC^{0,1}$ does not.  
\end{cor}
In light of the `constructive' status of the anti-Specker property (see Remark~\ref{floppy}), a more `constructive' proof of $\AS_{\net}$ is desirable. 
%Note that the following proof also goes through over \emph{intuitionistic logic}.
\begin{thm}
The system $\RCAo+\IND+\NFP_{0}+\HBU$ proves $\AS_{\net}$.
\end{thm}
\begin{proof}
Let $x_{d}:D\di \R$ be a net that is eventually bounded away from $I$, i.e.\
\be\label{wowa}\textstyle
(\forall x\in I)(\exists n^{0})(\exists d\in D)(\forall e\succeq d)(|x-x_{e}|\geq \frac{1}{2^{n}}).
\ee
By the definition of $[x](k)$ the $k$-th approximation of $x$, we have
\be\label{wowa2}\textstyle
(\forall x\in I)(\exists n^{0})\big[(\exists d\in D)(\forall e\succeq d)(|[x](2^{n+1})-x_{e}|\geq \frac{1}{2^{n+1}})\big].
\ee
The formula in square brackets in \eqref{wowa2} can be written $A(\overline{x}n)$ with only slight abuse of notation. 
To finish the proof, apply $\NFP_{0}$ to \eqref{wowa2} and follow the final part of the proof of Theorem~\ref{ochtendrood}.
\end{proof}
\begin{cor}\label{ochtendroodcor3}
The system $\RCAo+\IND+\NFP_{0}$ proves $\HBU\asa \AS_{\net}$.
\end{cor}
The previous result can be sharpened by introducing $\AS_{\net}^{-}$, which is $\AS_{\net}$ where the antecedent states the existence of $F:\R \di \R$ such that $F(x)$ is the number $\delta>0$ as in the first item of Definition \ref{ASS}, i.e.\ $F$ is a `bounded away' modulus.    
\begin{thm}\label{zeker}
The system $\RCAo+\IND$ proves $\HBU\asa \AS^{-}_{\net}$.
\end{thm}
\begin{proof}
The forward direction follows as in the second part of the proof of Theorem~\ref{ochtendrood}.  The reverse direction follows from the proof of Theorem \ref{kifu} by noting that $\min(1, \Psi(x))$ is the modulus required by the antecedent of $\AS^{-}_{\net}$.
\end{proof}
We note that $\AS_{\net}$ is remarkably robust, i.e.\ we do not know of any reasonable weakening.  % to the level of $\WHBU_{0}$.  
Similar to item \eqref{devil} in Section \ref{motiv}, there are basic spaces with 
a sequence that is bounded away from every point, but not from the entire space, i.e.\ the anti-Specker property for sequences does not capture the topology.     

\smallskip

We finish this section with a discussion of the provenance of the anti-Specker property.
To fully appreciate the following remark, one requires some basic familiarity with Bishop's \emph{Constructive Analysis} (\cite{bish1}) and the associated RM-development (\cite{ishi1}). 
\emph{Nonetheless}, all of the results in this section are part of \emph{classical} mathematics/logic and can be read without any knowledge of constructive mathematics. 
\begin{rem}\label{floppy}\rm
The sequential compactness of the unit interval is rejected in constructive mathematics as this property implies some fragment of the law of excluded middle (\cite{scrutihara}).  
A more constructive notion of sequential compactness was formulated in \cite{bribe} by considering the `antithesis' of \emph{Specker's theorem} (see \cite{brich}*{p.~58}); the latter theorem provides a recursive counterexample to the monotone convergence theorem.    
The associated general `anti-Specker property' was later introduced, intuitively expressing that if a sequence is eventually bounded away from any point in a space, then it is eventually bounded away (uniformly) from the entire space.  
The anti-Specker property (of certain spaces) is equivalent to (certain versions of) Brouwer's \emph{fan theorem}, a `semi-constructive' principle accepted in intuitionistic mathematics (see e.g.\ \cite{beeson1}).  
The classical contraposition of weak K\"onig's lemma is often referred to as `the' fan theorem (for decidable bars).      
\end{rem}

\subsubsection{Cauchy nets}\label{forfoxsake}
 In this section, we study basic theorems pertaining to \emph{Cauchy nets} (see e.g.\ \cite{ooskelly}*{p.\ 190}), defined as follows for $\R$.
 It goes without saying that such nets are the generalisation of the notion of Cauchy sequence to directed sets.  % as defined in Definition \ref{caucau}.  
\bdefi[Cauchy net]\label{caucau}
 A net $x_{d}:D\di \R$ is \emph{Cauchy} if $(\forall \eps>0)(\exists d\in D)(\forall e,f\succeq_{D} d)(|x_{e}-x_{f}|<\eps)$.
\edefi
Our motivation is two-fold: one one hand, the convergence of Cauchy \emph{sequences} in the unit interval is equivalent to $\ACA_{0}$ by \cite{simpson2}*{III.2.2}.
One the other hand, $\MCT_{\net}^{0}$ obviously follows from the combination of the following two theorems.  
\bdefi[$\CAU_{\net}$]
A Cauchy net in $[0,1]$ converges to a limit.
\edefi
\bdefi[$\CAU_{\net}'$]
An increasing net in $[0,1]$ is a Cauchy net.
\edefi
It is readily shown that $\Z_{2}^{\Omega}+\QFAC^{0,1}$ or $\RCAo+\IND+\NFP_{0}+\HBU$ proves $\CAU_{\net}$, while $\RCAo+\IND$ proves $\CAU_{\net}'$.  
As it turns out, both `Cauchy net theorems' have interesting properties, as follows.  
%
%\smallskip
%
%The convergence of Cauchy \emph{sequences} in the unit interval is equivalent to $\ACA_{0}$ by \cite{simpson2}*{III.2.2}.
%Firs of all, recall that $\IND$ is the induction schema for all formulas.  % as introduced in Section \ref{RM}.  %The system $\RCAo+\IND$ has the same first-order strength as $\ACA_{0}$.  
%We have the following.
\begin{thm}
The system $\RCAo+\IND+\CAU_{\net}$ proves $\HBU$.
\end{thm}
\begin{proof}
We make use of $(\exists^{2})\vee \neg(\exists^{2})$ as in the proof of Theorem \ref{merdes}.  The first part involving $\neg(\exists^{2})$ is identical.  
For the second part, let the net $x_{w}$ be as in the proof of Theorem \ref{merdescor}.  In case $x_{w}$ is Cauchy, the proof of the latter theorem
goes through.  In case $x_{w}$ is \emph{not} Cauchy, there is $k_{0}^{0}$ such that $(\forall d\in D)(\exists e,f\succeq_{D} d)(|x_{e}-x_{f}|\geq \frac{1}{2^{k_{0}}})$.
Applying the latter at most $2^{k}$ times, we obtain a finite sub-cover by the definition of $x_{w}$.  For this final step, it seems $\IND$ is needed. 
\end{proof}
Secondly, by Corollary \ref{koon}, the functional $\exists^{3}$ computes a realiser for the monotone convergence theorem for \emph{nets} via a term of G\"odel's $T$, and vice versa.
We now obtain similar results for $\CAU_{\net}$ and $\CAU_{\net}'$.  The latter is the most interesting.    

\smallskip

Now, a realiser for $\CAU_{\net}'$ takes as input $(D, \preceq_{D}, x_{d})$ and outputs $\Phi^{1\di 1}$ such that  $(\forall \eps>0)(\forall e,f\succeq_{D} \Phi(\eps))(|x_{e}-x_{f}|<\eps)$ if the inputs are as in $\CAU'_{\net}$. 
\begin{cor}\label{koonkoon}
A realiser for $\CAU_{\net}'$ together with $\exists^{2}$, computes $\exists^{3}$ via a term of G\"odel's $T$.  %, and vice versa.
\end{cor}
\begin{proof}
%For the `vice versa' direction, one uses the usual `interval halving technique' where $\exists^{3}$ is used to decide whether there is $d\in D$ such that $x_{d}$ is in the relevant interval. 
%For the other direction, 
Let $D$ be the set of finite sequences in Baire space and define $w\preceq_{D} v$ for $w, v\in D$ as $(\forall i<|w|)(\exists j<|v|)(w(i)=_{1}v(j))$ using $\exists^{2}$.  
Now fix $Y^{2}$ and define the net $x_{w}:D\di I$ as $1-\frac{1}{2^{|w|}}$ if $(\exists i<|w|)(Y(w(i))=0)$, and $0$ otherwise.  Clearly, $x_{w}$ is increasing, and let $\Phi$ be a modulus of Cauchy-ness.  
Note that $(\exists f\in \Phi(1/2))(Y(f)=0)\asa (\exists g^{1})(Y(g)=0)$, and we are done. 
\end{proof}
On a related note, to derive $\BW_{\net}$ from $\CAU_{\net}$, one requires $\COH_{\net}$, i.e.\ the statement \emph{any net in the unit interval contains a Cauchy sub-net}.  The associated property for \emph{sequences} is equivalent to $\COH$ from the RM zoo (see \cites{keuzer}).  
A realiser for $\COH_{\net}$ clearly computes $\exists^{3}$ by Corollary \ref{koonkoon}.  Moreover, in light of the proof of Corollary \ref{koonkoon}, a realiser for $\CAU_{\net}'$ also provides a witness $g^{1}$ such that $Y(g)=0$ if such exists, i.e.\ $\QFAC^{0,1}$ is involved, 
in contrast to Corollaries \ref{koon} and \ref{koonkoon2}.  

\smallskip

We now study realisers for $\CAU_{\net}$, which are tame by comparison (to the above).
A realiser for $\CAU_{\net}$ is a functional taking as input $(D, \preceq_{D}, x_{d})$ and outputting the limit $x=\lim_{d}x_{d}$ if the inputs satisfy the conditions of $\CAU_{\net}$. 
\begin{cor}\label{koonkoon2}
A realiser for $\CAU_{\net}$ together with $\exists^{2}$ computes $\exists^{3}$ via a term of G\"odel's $T$, and vice versa.
\end{cor}
\begin{proof}
For the `vice versa' direction, the limit exists and one uses the usual `interval halving technique' to locate it, where $\exists^{3}$ is used to decide whether there is a limit in the relevant half-interval. 
For the other direction, let $x_{w}$ be as in the proof of Corollary \ref{koonkoon}.  In case $(\forall g^{1})(Y(g)>0)$, $x_{w}$ is always $0$ and hence Cauchy.  
In case there is some $g^{1}_{0}$ such that $Y(g_{0})=0$, $x_{w}$ is also Cauchy, which is seen by considering long enough $w$ containing $g_{0}$.    
Clearly, $\lim_{d}x_{d}=1\asa(\exists g^{1})(Y(g)=0)$.  % and we are done. 
\end{proof}
%Finally, as will become clear in Section \ref{cauf}, $\CAU_{\net}'$ becomes quite powerful upon the introduction of a Cauchy \emph{modulus}.  

\subsubsection{Unordered sums}\label{cumsum}
We discuss \emph{unordered sums}, the generalisation of sums to possibly uncountable index sets (see e.g.\ \cite{zonbaardiemop}*{\S5.2}, \cite{halmospalmos}*{Ch.~1, \S7}. \cite{ooskelly}*{p.\ 76}, \cite{ooskelly2}*{Ch.\ 0}, or \cite{thom2}*{\S3.3}).  
Historically, the study of such sums by Moore in \cite{moorelimit1} was the first step toward the Moore-Smith theory of convergence in \cite{moorsmidje}.  Moreover, unordered sums allow for an alternative formulation of measure theory (see \cite{ooskelly}*{p.\ 79}).  

\smallskip

For $D\subseteq \N^{\N}$ and any $a_{d}:D\di \R$, we want to provide meaning to `the uncountable sum $\sum_{d\in D}a_{d}$'.
To this end, let $D^{*}$ be the set of finite subsets of elements of $D$, which is a directed set if $\preceq_{D^{*}}$ is inclusion on $D^{*}$.   
The net $(\sum_{i\in d}a_{i}):D^{*}\di \R $ then behaves as in the following (most) basic permutation theorem.  
\bdefi[$\PERMU$]
For any $D\subseteq \N^{\N}$ and any $a_{d}:D\di \R$, if 
\be\label{caucon}\textstyle
(\forall \eps>0)(\exists d \in D^{*})(\forall e, f \succeq_{D^{*}} d )\big( \big|\sum_{i\in e}a_{i}  -\sum_{j\in f}a_{j}|<\eps  \big),
\ee
then the net $(\sum_{i\in d}a_{i}):D^{*}\di \R $ converges to some $a\in \R$.   % (in the sense of nets). % to $\lim_{w}a_{w}$.
\edefi
The above limit $a\in \R$ bestows meaning onto `the uncountable sum $\sum_{d\in D}a_{d}$'.  A realiser for $\PERM$ takes as input $D\subseteq \N^{\N}$ and $a_{d}:D\di \R$ and outputs the limit $a\in \R$ if \eqref{caucon} is satisfied. 
Following the definitions in \cite{simpson2}*{V.2}, a realiser for $\ATR_{0}$ is any functional that outputs $Y$ as in $H_{\theta}(X, Y)$ on input 
$f^{1}$ such that $\theta(n, Z)\equiv (\forall k^{0})(f(\overline{Z}k, n)=0)$ and any countable well-ordering $X$. 

\begin{thm}\label{koren}
A realiser for $\PERMU$ computes $\exists^{2}$ and a realiser for $\ATR_{0}$ via a term of G\"odel's $T$. 
\end{thm}
\begin{proof}
First of all, to obtain $\exists^{2}$, consider $f^{1}$ and define the sequence $a_{n}$ as $1$ if $n$ is the least number such that $f(n)=0$, and $0$ otherwise.  
Clearly, $a_{n}$ satisfies \eqref{caucon} and the limit $a\in \R$ is such that $a=_{\R}1\asa (\exists n^{0})(f(n)=0)$.  

\smallskip

Secondly, consider \cite{simpson2}*{V.5.2} which shows that $\ATR_{0}$ is equivalent to
\be\label{happart}
(\forall n^{0})(\exists \textup{ at most one }X)\varphi(n, X)\di (\exists Z)(\forall m^{0})(m\in Z\asa (\exists X)\varphi(m, X)), 
\ee
for any arithmetical $\varphi$ and over $\RCA_{0}$.  The proof of \cite{simpson2}*{V.5.2} yields that a realiser for $\ATR_{0}$ is readily defined in terms of any functional that outputs $Z$ as in \eqref{happart} on input 
$f^{1}$ such that $\varphi(n, X)\equiv (\forall k^{0})(f(\overline{X}k, n)=0)$ satisfying the uniqueness in \eqref{happart}. 
Now let $D$ be Cantor space, fix some $n^{0}$, and define $a_{d}:D\di \R$ as $1$ if $\varphi(n, d)$, and zero otherwise.  
Clearly, $a_{d}$ satisfies \eqref{caucon} and the limit $a\in \R$ is such that $a=_{\R}1\asa (\exists d\in D)\varphi(n, d)$, if $\varphi$ satisfies uniqueness as in \eqref{happart}.  
\end{proof}
%A result not involving realisers will be obtained in Corollary \ref{konky}.
 
\subsection{Compactness and nets of functions}\label{shasup2}
In this section, we study theorems pertaining to nets \emph{of continuous functions}, like Dini's theorem (Section \ref{sha4}) and Arzel\`a's theorem (Section \ref{shaka}).  
It goes without saying that for nets \emph{of functions} $f_{d}:(D\times [0,1])\di \R$, properties of $f_{d}(x)$ like continuity pertain to the variable $x$, while the net is indexed by $d\in D$. 
For instance, an increasing net is as follows. 
\bdefi[Increasing net]\label{inet}
A net of functions $f_{d}:(D\times I)\di \R$ is \emph{increasing} if $a\preceq b$ implies $f_{a}(x)\leq_{\R} f_{b}(x) $ for all $x\in I$ and $a,b\in D$.
\edefi
We remind the reader that we restrict ourselves to nets that are indexed by subsets of Baire space.  

\subsubsection{Dini's theorem}\label{sha4}
We study a version of Dini's theorem \emph{for nets}, which may be found in many places: \cite{demofte,naim, bartlecomp, kupkaas,IDA, tomaat, wolk, ooskelly, moorsmidje}.  

\smallskip

By Corollary \ref{formi}, the following version of Dini's theorem for nets is equivalent to $\HBU$.  
We say that $f_{d}:(D\times I)\di \R$ converges \emph{uniformly} if the net $\lambda d.f_{d}(x)$ converges, and $d_{0}\in D$ as in Definition \ref{convnet} does not depend on the choice of $x\in I$.  
%We also show that the logical strength of Dini's theorem depends greatly on the framework. 
\bdefi[$\DIN_{\net}$]
For continuous $f_{d}:(D\times I)\di \R$ forming an increasing net and converging to continuous $f:I\di \R$, the convergence is uniform.
\edefi
\begin{thm}\label{merdes3}
The system $\RCAo+\DIN_{\net}$ proves $\HBU$.
\end{thm}
\begin{proof}
The `classical' Dini's theorem (for sequences) is equivalent to $\WKL$ by \cite{diniberg}*{Theorem 21}, i.e.\ we have access to the latter.  
Now, in case $\neg(\exists^{2})$, all functions on $\R$ are continuous by \cite{kohlenbach2}*{Prop.\ 3.12}, and $\HBU$ reduces to $\WKL$ by \cite{kohlenbach4}*{\S4}.
We now prove $\HBU$ in case $(\exists^{2})$, which finishes the proof using the law of excluded middle.  

\smallskip

%Suppose $\neg \HBU$ and let $\Psi:I\di \R^{+}$ be such that the associated canonical cover does not have a finite sub-cover.  
Fix some $\Psi:I\di \R$ and let $D$ be the set of finite sequences of reals in $I$ and define `$v \preceq w$' for $w, v\in D$ if $(\forall i<|v|)(\exists j<|w|)(v(i)=_{\R} w(j))$, i.e.\ as in the proof of Theorem \ref{merdescor}.
Now define $f_{w}:I\di \R$ as follows:  if $w=\langle x\rangle$ for some $x\in I$, then $f_{w}$ is $0$ outside of $I_{x}^{\Psi}$, while inside the latter, $f_{w}(x)$ is the piecewise linear function that is $1$ at $x$, and $0$ in $x\pm\Psi(x)$.
Note that these objects have 
the required basic properties (of directed set, net, et cetera).  
Moreover, $f_{w}$ is also increasing (in the sense of nets) and converges to the constant one function (in the sense of nets), as for any $v\succeq \langle x \rangle$, we have $f_{v}(x)=1$.   
Now apply $\DIN_{\net}$ and conclude that the convergence is uniform.  Hence, applying the erstwhile theorem for $\eps=1/2$, there is $w_{0}$ such that for all $x\in I$, $f_{w_{0}}(x)>0$.  
However, the latter implies that every $x\in I$ is in $\cup_{i<|w_{0}|}I_{w_{0}(i)}^{\Psi}$, i.e.\ we found a finite sub-cover, yielding $\HBU$ for $\Psi$.  %This contradiction yields $\HBU$.  % and we are done.   
\end{proof}
Since Dini's theorem is equivalent to $\WKL$ in classical RM, we expect the following result. 
Using the continuity properties of the functions in the net, one can get by with $\QFAC^{0,1}$, but the latter axiom does seem essential.  
Moreover, using the above `excluded middle' trick, one could omit $(\exists^{2})$.  %as nets (essentially) reduce to sequences if all functions on Baire space are continuous. 
\begin{cor}\label{formi}
The system $\ACAo+\QFAC^{1,1}$ proves $\HBU\asa \DIN_{\net}$.
\end{cor}
\begin{proof}
We only have to prove the forward direction.  As in the usual proof of Dini's theorem, we may assume that the net $f_{d}$ is decreasing and converges pointwise to the constant zero function.
Fix $\eps_{0}>0$ and apply $\QFAC^{1,1}$ to $(\forall z\in I)(\exists d\in D)(0\leq f_{d}(z) <\eps_{0})$, to obtain $\Phi^{1\di 1}$ yielding $d\in D$ from $z\in I$.  Since $\lambda x.f_{\Phi(z)}(x)$ is continuous for any fixed $z$, 
$(\exists^{2})$ yields a modulus of continuity $g$ as in the proof of \cite{kohlenbach4}*{Prop.\ 4.7}, i.e.\ we have:
\be\label{contjer}\textstyle
(\forall \eps>0)(\forall x, y\in I)(|x-y|<g(x, \eps,z)\di |f_{\Phi(z)}(x)-f_{\Phi(z)}(y)|<\eps),
\ee
for all $z\in I$.  Define $\Psi:I \di \R^{+}$ as $\Psi(x):=g(x,\eps_{0},x)$ and note that $(0\leq f_{\Phi(x)}(y) <\eps_{0})$ for all $y\in I_{x}^{\Psi}$ by \eqref{contjer} and the definition of $\Phi$. 
Now let $y_{1}, \dots, y_{k}$ be the associated finite sub-cover provided by $\HBU$. 
By item \eqref{bulk} of Definition \ref{nets}, there is $d_{0}\in D$ such that $d_{0}\succeq \Phi(y_{i})$ for all $i\leq k$.  Since $f_{d}$ is a decreasing net and $[0,1]\subset \cup_{i\leq k}I_{y_{i}}^{\Psi}$, we have $(0\leq f_{d}(y) <\eps_{0})$ for all $y\in I$ and $d\succeq d_{0}$, i.e.\ uniform convergence as required.
\end{proof}
A detailed study of the proof of \cite{kohlenbach4}*{Prop.\ 4.10} shows that one can avoid the use of $(\exists^{2})$ to obtain the modulus of continuity in the previous proof; indeed, by the aforementioned result, it suffices to have $\WKL$, which follows from $\HBU$.
We could weaken the conclusion of Dini's theorem to \emph{convergence in measure} or convergence of integrals, and the resulting theorems would be equivalent to \emph{weak compactness} as in Vitali's covering theorem; see \cite{dagsamVI} for details.

%We do establish Corollary \ref{merdes2cort}.
%\begin{cor}\label{merdes2cort2}
%The system $\RCAo+\WWKL+\QFAC^{1,1}$ proves $\MCT_{\net}\asa \WHBU_{0}$.
%\end{cor}
%\begin{proof}
%We only need to prove the reverse direction.  
%The proof of Corollary \ref{formi} can be used if we replace $\HBU$ by $\WHBU_{0}$ and note that 
%for \emph{bounded} functions on $[0,1]$, the Riemann integral is invariant to changes of the integrand in small areas.   
%\end{proof}

\subsubsection{Arzel\`a's theorem}\label{shaka}
We show that Arzel\`a's theorem  for nets (see\footnote{Note that \cite{cassacassa} includes an historical overview pertaining to Arzel\`a's theorem (for nets).} e.g.\ \cites{brace4imp,bartlecomp,moorsmidje,cassacassa}) implies $\HBU$.
This theorem deals with \emph{quasi-uniform} convergence, a notion apparently first introduced by Arzel\`a himself in \cite{arse}*{Def.\ 2.1}.
\bdefi[Quasi-uniform convergence of nets]\label{qum}
A net $f_{d}:(D\times I)\di \R$ \emph{converges quasi-uniformly} to $f:I\di \R$ if $f$ is the limit of the net $f_{d}$ and
\[
(\forall \eps>0, d\in D)(\exists d_{0}, \dots, d_{k}\succeq d)(\forall x\in I)(\exists j\leq k)(|f_{d_{j}}(x)-f(x)|<\eps).
\]
\edefi
Arzel\`a's theorem now has the following generalisation to nets.  
\bdefi[$\ARZ_{\net}$]
For continuous $f_{d}:(D\times I)\di \R$ forming a net convergent to a continuous $f:I\di \R$, the convergence is quasi-uniform.
%\begin{enumerate}
%\item for every $\eps>0$ and $d\in D$, 
%there are $d_{0}, \dots, d_{k}\in D$ such that $d_{i}\succeq d$ for all $i\leq k$ and for each $x\in [0,1]$, there is $j\leq k$ such that $|f_{d_{j}}(x)-f(x)|<\eps$.
%\end{enumerate}
\edefi
\begin{thm}\label{merdes4}
The system $\RCAo+\ARZ_{\net}$ proves $\HBU$.
\end{thm}
\begin{proof}
The proof of the theorem is similar to the proof of Theorem \ref{merdes3}.  Indeed, for $f_{w}$ as in the latter, quasi-uniform convergence for $\eps=1/2$ and $d=\langle 0\rangle$, yields $w_{0}, \dots, w_{k}\succeq \langle 0\rangle$ such that for each $x\in [0,1]$, there is $j\leq k$ such that $|f_{w_{j}}(x)-f(x)|<1/2$.  As in the proof of Theorem \ref{merdes3}, this implies that $w_{0}*\dots *w_{k}$ yields a finite sub-cover of the canonical cover associated to $\Psi$, and we are done. 
\end{proof}
The Ascoli-Arzel\`a theorem for nets (see e.g.\ \cite{degou}*{p.\ 247}) similarly implies $\HBU$, since it implies the Bolzano-Weierstrass theorem for nets.  
%see https://math.stackexchange.com/questions/608799/arzela-ascoli-net-question

\smallskip

We finish this section with a conceptual remark regarding quasi-convergence.  
\begin{rem}[Quasi-convergence and the weak-$*$-topology]\label{weakling}\rm
Dual spaces and the associated weak-$*$-topology are studied in RM (see e.g.\ \cite{simpson2}*{X.2}).   
Moreover, it has been known for more than half a century that quasi-uniform convergence for nets is related to the weak and weak-$*$-topologies (see \cites{brace4imp, bartlecomp, brace5imp, nieuweman}).
For instance, quasi-convergence for nets yields an equivalent formulation of the weak-$*$-topology for a large class of spaces by \cite{brace5imp}*{Theorem 3.1}.
In this light, the study of net convergence, and $\ARZ_{\net}$ in particular, in (higher-order) RM is quite natural. 
\end{rem}
%\subsubsection{Ascoli-Arzel\`a theorem}
%\begin{thm}[$\AA$]
%If $f_{d}:(D\times I)\di I$ is an equi-continuous net converging to a continuous function $f:I\di I$, then there is a uniformly convergent sub-net.
%\end{thm}
%
%
%HERE

\section{Main results II}\label{espero}
As suggested by item \ref{devil} in Section \ref{motiv}, sequences do not suffice for describing topologies in general, and nets are needed instead. 
Intuitively speaking, we show in this section that even for spaces like $\R$ where sequences do suffice to describe the topology (say working in $\ZFC$), the absence of countable choice (say over $\RCAo$) implies that sequences no longer suffice to describe the topology, but nets do suffice.

\smallskip

On a historical note, Root, a student of E.H.\ Moore, already studied when limits from Moore's \emph{General Analysis} (\cite{moorelimit2}) can be replaced by limits given by sequences (\cite{rootsbloodyroots}).  Thus, the idea of replacing nets by sequences goes back more than a century.    

%As will become clear, countable choice plays an essential role in all this.     
 
\subsection{Introduction}\label{nintro}
Nets are generalisations of sequences, and it is therefore a natural question `how hard' it is to replace the former by the latter.  
In Section~\ref{NAP}, we study such an `sequentialisation' principle, called $\SUB$, from Bourbaki's \emph{general topology} (\cite{gentoporg2}); 
we show that despite its highly elementary nature, $\SUB$ implies the \emph{Lindel\"of lemma} for $\R$, a close relative of $\HBU$.
We also show that $\SUB_{0}$, a special case of $\SUB$, is equivalent to $\QFAC^{0,1}$, assuming (natural) extra axioms.
Thus, in the absence of countable choice, nets are more general than sequences in terms of convergence on $\R$.
In general, it should be noted that such sequentialisation theorems are only valid/possible for first-countable spaces.  
%The proof of $\SUB_{0}\asa \QFAC^{0,1}$ in turn gives rise to a remarkable result: the generalisation of $\ADS$ from the RM zoo (see \cite{dsliceke}) to linear orders on Baire space is not provable in $\ZF$.

\smallskip

Inspired by the previous paragraph, it is a natural question whether `upgrading' sequential continuity with nets has any noteworthy effects.  
In Section \ref{NAP2}, we prove the \emph{local} equivalence of the resulting `net-continuity' and `epsilon-delta' continuity in $\RCAo$.  
Note that the local equivalence between \emph{sequential} continuity and epsilon-delta continuity cannot be proved in $\ZF$ (\cite{rimjob}), while $\QFAC^{0,1}$ suffices (\cite{kohlenbach2, kohlenbach4}).  
%The latter fragment of the Axiom of Choice is equivalent to a `strong' version of the Lindl\"of lemma, as discussed in \cite{dagsamV}*{\S5}, explainin.  

\smallskip

Similarly, we show in Section \ref{neclo} that $\R$ is a sequential space, i.e.\ that `sequentially closed' sub-sets of $\R$ are closed, over $\RCAo+\QFAC^{0,1}$;  this result cannot be proved in $\ZF$ by \cite{heerlijkheid}*{p.~73}, i.e.\ $\QFAC^{0,1}$ is essential, as in the case of sequential continuity.  
By contrast, we also prove that the generalisation from sequences to nets does not require the Axiom of Choice: `net-closed' sets are closed over $\RCAo$. 
%
%\smallskip
%
%Moreover, we show in Section \ref{pitche} that nets obviate the need for strong comprehension axioms \emph{and} countable choice in the proof of \emph{Pincherele's theorem}, one of the first `local-to-global' principles (\cite{tepelpinch}*{p.\ 67}).

\smallskip

We stress that the previous is not merely \emph{spielerei}: the definition of closed sets in \cite{gieren} and the definition of continuity in \cites{degou, gieren} are given \emph{in terms of nets}.  
In other words, nets are central to domain theory and are used to define the notions of closed set and continuous function.  
Moreover, our results show that using nets instead of sequence obviates the need for the Axiom of Choice, a foundational concern in domain theory by the quotes from Section \ref{intro}. 
%\begin{quote}
%Turning to foundations, we feel that the necessity to choose chains where directed subsets are naturally available (such as in function spaces) and thus to rely on the Axiom of Choice without need, is a serious stain on this approach. (\cite{aju}*{\S2.2.4}).
%\end{quote}
We remind the reader that we restrict ourselves to nets indexed by Baire space.  
\subsection{Nets and sequentialisation}\label{NAP}
By the above, basic theorems regarding nets imply $\HBU$ and therefore require rather strong comprehension axioms for a proof.  
In line with the coding practise of RM, one may therefore want to replace limits involving nets by `countable' limits, i.e.\ if a net converges to some limit, then there should be a \emph{sequence} in the net that also converges to the same limit.  
In this section, we show that such `sequentialisation' theorems imply $\QFAC^{0,1}$ (Theorem \ref{himenez1337}) and the Lindel\"of lemma (Theorem \ref{himenez}), and obtain a nice spin-off result (Theorem~\ref{wttf}) regarding the RM zoo (\cite{damirzoo}).   
In general, it should be noted that such sequentialisation theorems are only valid/possible for first-countable spaces.  

\medskip

First of all, we show that even an highly elementary version of the aforementioned sequentialisation theorem implies the \emph{Lindel\"of lemma} for $\R$ from \cite{dagsamIII}, as follows.
\bdefi[$\LIND$] 
For every $\Psi:\R\di \R^{+}$, there is a sequence of open intervals $\cup_{n\in \N}(a_{n}, b_{n})$ covering $\R$ such that $(\forall n \in\N)(\exists x \in \R)[(a_{n}, b_{n}) = I_{x}^{\Psi} ]$.  % and $\cup_{n}(a_{n},b_{n})$ covers $\R$.
\edefi
\noindent
Lindel\"of proved the {Lindel\"of lemma} in 1903 (\cite{blindeloef}), while Young and Riesz proved a similar theorem in 1902 and 1905 (\cite{manon, youngster}); $\LIND$ expresses that an open cover of $\R$ 
has a countable sub-cover, and is very close to Lindel\"of's original lemma\footnote{Lindel\"of formulates his lemma in \cite{blindeloef}*{p.\ 698} as follows:  
\emph{Let $P$ be any set in $\R^{n}$ and construct for every point of $P$ a sphere $S_{P}$ with $x$ as center and radius $\rho_{P}$, where the latter can vary from point to point; there exists a countable infinity $P'$ of such spheres such that every point in $P$ is interior to at least one sphere in $P'$.}
%\emph{Soit \textsf{\textup{(P)}} un ensemble quelconque situ\'e dans l'espace $\R^{n}$ et, de chaque point $\textsf{\textup{P}}$ comme centre, construisons une sph\`ere $\textsf{\textup{S}}_{\textsf{\textup{P}}}$ d'un rayon $\rho_{\textsf{\textup{P}}}$ qui peut varier de l'un point
%\`a l'autre; il existe une infinit\'e d\'enombrable de ces sph\`eres de telle sorte que tout point de l'ensemble donn\'e
%soit int\'erieur \`a au moins l'une d'elles}.
\label{cardrar}}.  

\smallskip

By \cite{dagsamIII}*{Theorem 3.13}, $\HBU$ is equivalent to $[\LIN+\WKL]$, i.e.\ $\LIN$ is extremely hard to prove, while a connection between $\LIND$ and some theorem about nets is expected by the previous. 
In particular, $\SUB$ fulfils that role by Theorem \ref{himenez}.
\bdefi[\SUB]
For $f_{d}:(D\times I)\di \R$ an increasing net of continuous functions converging to continuous $f=\lim_{d}f_{d}$, there is $\Phi:\N\di D$ such that $f_{\Phi(n)}$ is increasing (in the variable $n$) and $\lim_{n\di \infty}f_{\Phi(n)}=f$.  
%If the net converges pointwise to a function g in C(X), there exists an increasing (respectively, decreasing) sequence of functions (fn) belonging to the net which converges pointwise to g.
\edefi
Note $\SUB$'s narrow scope, i.e.\ it only seems to apply to $\DIN_{\net}$ and $\MCT_{\net}$.
Nonetheless, $\SUB$ occurs in Bourbaki's general topology, namely \cite{gentoporg2}*{p.\ 337}.  
\begin{thm}\label{himenez}
The system $\RCAo+\SUB$ proves $\LIND$.
\end{thm}
\begin{proof}
In case $\neg(\exists^{2})$, all functions on $\R$ are continuous by \cite{kohlenbach2}*{Prop.\ 3.12}. 
The countable sub-cover required for $\LIN$ is then given by $\Q$.  
In case $(\exists^{2})$, suppose $\neg \LIN$ and let $\Psi:\R\di \R^{+}$ be such that the associated canonical cover does not have a countable sub-cover.   
We let $D$ be the set of sequences of real numbers and we define the relation between such sequences `$\lambda n.x_{n}\preceq \lambda n.y_{n}$' as 
\be\label{corg}
(\forall n\in \N)(\exists m\in \N)(x_{n}=_{\R}y_{m}) %\wedge \cup_{n\in \N}I_{x_{n}}^{\Psi} \subset \cup_{n\in \N}I_{y_{n}}^{\Psi}.
\ee
Clearly, this relation yields a directed set.  
Now define $f_{w}:I\di \R$ as follows:  If $w=\langle x\rangle$ for some $x\in I$, then $f_{w}$ is $0$ outside of $I_{x}^{\Psi}$, while inside the latter, $f_{w}(x)$ is the piecewise linear function that is $1$ at $x$, and $0$ in $x\pm\Psi(x)$.
If $w$ is a sequence, then $f_{w}(x)=\sup_{n\in \N}f_{\langle w(n)\rangle}(x)$.  Clearly, $f_{w}$ is increasing (in the sense of nets) and converges to the constant one function (in the sense of nets), as for any $v\succeq (x, x, x, \dots)$, we have $f_{v}(x)=1$.  
Now let $\Phi^{0\di 1}$ be as in $\SUB$ and create a `master sequence' of reals $\lambda n.z_{n}$ containing the sequences $\Phi(1)$, $\Phi(2)$, $\Phi(3)$ et cetera.  
By $\SUB$, for any $x\in \R$, there is $n_{0}$ such that $|f_{\Phi(n_{0})}(x)-1|<\frac{1}{2}$, i.e.\ there is $m\in \N$ such that $x\in I_{\Phi(n_{0})(m_{0})}^{\Psi}$.  
Since the real $\Phi(n_{0})(m_{0})$ is part of the master sequence $z_{n}$, we obtain $\LIND$.
\end{proof} 
It is possible to obtain an equivalence in the previous theorem by considering the more general `Borel-Schoenflies' version of $\LIN$ from \cite{dagsamV}*{\S5.3} and $\QFAC^{0,1}$ for real quantifiers.
The proofs are however similar, so we do not go into details.  We do prove the equivalence between $\QFAC^{0,1}$ and a special case of $\SUB$ as follows. 
\bdefi[$\SUB_{0}$]
For $x_{d}:D\di I$ an increasing net converging to $x\in I$, there is $\Phi:\N\di D$ such that $\lambda n.x_{\Phi(n)}$ is increasing and $\lim_{n\di \infty}x_{\Phi(n)}=_{\R}x$.  
%If the net converges pointwise to a function g in C(X), there exists an increasing (respectively, decreasing) sequence of functions (fn) belonging to the net which converges pointwise to g.
\edefi
Recall that $\IND$ is the induction schema for all formulas of $\L_{\omega}$.  
%The system $\RCAo+\IND$ has the same first-order strength as $\ACA_{0}$.  
\begin{thm}\label{himenez1337}
The system $\RCAo+\IND$ proves $\SUB_{0} \di\QFAC^{0,1}$.
\end{thm}
\begin{proof}
In case $\neg(\exists^{2})$, all functions on Baire space are continuous by \cite{kohlenbach2}*{Prop.~3.7}, and $\QFAC^{0,1}$ clearly reduces to $\QFAC^{0,0}$, included in $\RCAo$.  For the case $(\exists^{2})$, note that we also have $(\mu^{2})$.  
Let $\b^{1\di 1^{*}}$ be the inverse of a pairing function defined as $|\b(f)|=f(0)+1$ and $\b(f)(i)$ for $i<|\b(f)|$ is the sequence $f(1+i), f(1+i+|\b(f)|), f(1+i+2|\b(f)|), \dots$, which is definable in $\RCAo$.
Fix some $F^{(0\times 1)\di 0}$ satisfying the antecedent of $\QFAC^{0,1}$, i.e.\ $(\forall n^{0})(\exists f^{1})(F(n, f)=0)$, and use $\IND$ to prove $(\forall n^{0})(\exists f^{1})\underline{(\forall i\leq n)(F(i, \b(\langle n \rangle * f)(i))=0)}$.  
The underlined formula is also written `$G(n, f)=0$' and if there is $f_{0}^{1}$ such that $(\forall n^{0})(G(n,f_{0})=0)$, then $Y(n):= \b(\langle n\rangle *f_{0})(n)$ is as required for the consequent of $\QFAC^{0,1}$.

\smallskip

Otherwise, i.e.\ in case $(\forall f^{1})(\exists n^{0})(G(n,f)\ne 0)$, define the set $D:=\{f^{1}:(\exists n^{0} )G(n, f)=0  \}$ and define the predicate `$\preceq$' as: $f\preceq g$ if and only if 
\be\label{cruzie}
(\mu n)(G(n, f)\ne 0)\leq (\mu m)(G(m, g)\ne 0),
%(\mu n)(\forall i\leq n)(F(i, f\b(f)(i)=0)\leq (\mu m)(\forall j\leq m)(F(j,\b(g)(j))=0).
\ee
which is well-defined by assumption.  
Note that $D$ with $\preceq$ forms a directed set by assumption. 
Define the increasing net $x_{d}:= 1-2^{-(\mu n)(G(n, d)\ne0)}$ and note that $\lim_{d}x_{d}=1$ by assumption and \eqref{cruzie}.  By $\SUB_{0}$, there is some $\Phi^{0\di 1}$ such that $\lim_{n\di \infty} x_{\Phi(n)}=1$, i.e.\ 
$(\forall \eps>0)(\exists m^{0})(\forall k^{0}\geq m)(|x_{\Phi(k)}-1|<\eps)$, and use $\mu^{2}$ to find $\Psi^{2}$ computing such $m^{0}$ from $\eps$.  
Then the functional $Y(n):= {\Phi(\Psi(\frac{1}{2^{n+1}}))}$ provides the witness as required for the conclusion of $\QFAC^{0,1}$.  
\end{proof} 
Let $\ADS$ be the $\L_{2}$-sentence from the RM zoo (see \cite{dsliceke}*{Def.\ 9.1}) that every infinite linear order has an infinite ascending or descending sequence.
\begin{cor}\label{zienoudoor}
The system $\RCAo+\IND+\ADS$ proves $\SUB_{0} \asa \QFAC^{0,1}$.
\end{cor}
\begin{proof}
We only need to prove the reverse implication.  To this end, let $x_{d}:D\di I$ be an increasing net converging to some $x\in I$.  This convergence trivially implies:
\be\label{zosimpel}\textstyle
(\forall k\in \N)(\exists d\in D)(|x-x_{d}|<\frac{1}{2^{k}}), 
\ee
and applying $\QFAC^{0, 1}$ to \eqref{zosimpel} yields $\Phi:\N\di D$ such that \emph{the sequence} $\lambda k^{0}.x_{\Phi(k)}$ also converges to $x$ as $k\di \infty$.  
Since $\ADS$ is equivalent to the statement that every sequence in $\R$ has a monotone sub-sequence (see \cite{kruiserover}*{\S3}), $\SUB_{0}$ now follows. 
\end{proof}
As is clear from the previous two proofs, it is straightforward to omit the two occurrences of `increasing' in $\SUB_{0}$.
It is a natural RM-question, posed previously by Hirschfeldt (see \cite{montahue}*{\S6.1}), 
whether the extra axioms are needed in the base theory of Corollary \ref{zienoudoor}.    

\smallskip

Finally, inspired by the proof of Theorem \ref{himenez1337}, we show that $\ADS$ generalised to uncountable linear orders\footnote{The prototypical uncountable linear order is given by $(\R, \leq_{\R})$, where `$x=_{\R}y$' is equivalent to $x\leq_{\R}y\wedge y\leq_{\R} x$ (see \cite{simpson2}*{II.4}).  Hence, we implicitly assume that an uncountable linear order $(D, \leq_{D})$ has an equality relation $x =_{D} y$ equivalent to $x\leq_{D}y \wedge y\leq_{D}x$.} is not provable in $\ZF$.
We restrict ourselves as in Definition \ref{strijker}, i.e.\ a linear order $(D, \leq_{D})$ is given by a subset $D$ of Baire space with a binary relation $\leq_{D}$ thereon, satisfying the usual properties.  
\bdefi[$\ADS_{2}$]
For $(D, \leq_{D})$ an infinite linear order, there is an ascending or descending sequence $x_{n}$, i.e.\ $(\forall n\in \N)(x_{n}<_{D}x_{n+1})\vee(\forall n\in \N)(x_{n}>_{D}x_{n+1})$.
\edefi
\begin{thm}\label{wttf}
The system $\RCAo+\IND+\ADS_{2}$ proves $\QFAC^{0,1}$.
\end{thm}
\begin{proof}
In case $\neg(\exists^{2})$, all functions on Baire space are continuous by \cite{kohlenbach2}*{Prop.~3.7}, and $\QFAC^{0,1}$ clearly reduces to $\QFAC^{0,0}$, included in $\RCAo$.  For the case $(\exists^{2})$, note that we also have $(\mu^{2})$.  
Let $\b^{1\di 1^{*}}$ be as in the proof of Theorem \ref{himenez1337}.  
% the inverse of a pairing function defined as $|\b(f)|=f(0)+1$ and $\b(f)(i)$ for $i<|\b(f)|$ is the sequence $f(1+i), f(1+i+\b(f)), f(1+i+2\b(f)), \dots$, which is definable in $\RCAo$.
Fix some $F^{(0\times 1)\di 0}$ satisfying the antecedent of $\QFAC^{0,1}$, i.e.\ $(\forall n^{0})(\exists f^{1})(F(n, f)=0)$, and use $\IND$ to prove $(\forall n^{0})(\exists f^{1})\underline{(\forall i\leq n)(F(i, \b(\langle n \rangle * f)(i))=0)}$.  
The underlined formula is also written `$G(n, f)=0$' and if there is $f_{0}^{1}$ such that $(\forall n^{0})(G(n,f_{0})=0)$, then $Y(n):= \b(\langle n\rangle *f_{0})(n)$ is as required for the consequent of $\QFAC^{0,1}$.

\smallskip

In case $(\forall f^{1})(\exists n^{0})(G(n,f)\ne 0)$, 
an equality on $D=\N^{\N}$ is as follows: `$f=_{D}g$' is $(\mu n)(G(n, f)\ne 0)=_{0} (\mu m)(G(m, g)\ne 0)$.     
Now define the order `$f\leq_{D} g$' as $(\mu n)(G(n, f)\ne 0)\leq_{0} (\mu m)(G(m, g)\ne 0)$.  Clearly, $(D, \leq_{D})$ is a linear order and applying $\ADS_{2}$, there is an ascending sequence $x_{n}$ in $D$, i.e.\ $x_{n}<_{D} x_{n+1}$ for all $n\in \N$.
Since $(\mu m)G(m, x_{n+1})\geq n+2$, we have $G(n,x_{n+1})=0$, as required.  % and we are done. 
\end{proof}
\begin{cor}
The system $\ZF$ cannot prove $\ADS_{2}$.
\end{cor}
In conclusion, we note that the power of $\ADS_{2}$ seems to stem from the ordering relation $\leq_{D}$: the latter is a true third-order object, as is clear from the proof.  
Moreover, if one demands that the relation `$f\leq_{D}g$' is given by $\varphi(f, g)$ for some $\varphi\in \L_{2}$ (and the same for `$f\in D$'), the associated restriction of $\ADS_{2}$ is of course provable using some fragment of dependent choice in $\Z_{2}$ (\cite{simpson2}*{VII.6.1}).  

\subsection{Nets and continuity}\label{NAP2}
We establish that `net-continuity' as in Definition \ref{netcont} and `epsilon-delta' continuity are \emph{locally} equivalent over $\RCAo$.
As discussed in \cite{kohlenbach2}*{Rem.\ 3.13}, $\ZF$ cannot prove the local\footnote{By \cite{kohlenbach2}*{Prop.\ 3.6}, $\RCAo$ can prove the \emph{global} equivalence of sequential continuity and epsilon-delta continuity on $\N^{\N}$, i.e.\ when those continuity properties hold everywhere on the latter.} equivalence of \emph{sequential} and epsilon-delta continuity (\cite{rimjob}), while $\QFAC^{0,1}$ suffices to establish the general case.  
\bdefi[Net-continuity]\label{netcont}
A function $f:\R\di \R$ is \emph{net-continuous} at $x\in \R$ if for any net $x_{d}$ in $\R$ converging to $x$, the net $f(x_{d})$ also converges to $f(x)$.
\edefi
Note that net-continuity is equivalent to the topological definition of continuity by \cite{zonderfilter}*{Example 2.7}. 
As it happens, the definition of continuity in \cite{gieren}*{p.\ 45} \emph{is} the definition of net-continuity.
It should be noted that \emph{Scott continuity} is a much more important/central notion than net-continuity in domain theory.   
\begin{thm}[$\RCAo$]\label{zienoudaar}
For any $f:\R\di \R$ and $x\in \R$, the following are equivalent:
\begin{enumerate}
 \renewcommand{\theenumi}{\alph{enumi}}
\item the function $f:\R\di \R$ is net-continuous at $x$,\label{kermit}
\item $(\forall \eps>0)(\exists \delta>0)(\forall y\in \R)(|x-y|<\delta\di |f(x)-f(y)|<\eps)$.\label{anti}
\end{enumerate}
\end{thm}
\begin{proof}
The implication $\eqref{anti}\di \eqref{kermit}$ is immediate.  For the remaining implication, note that in case of $\neg(\exists^{2})$, all $f:\R\di \R$ are continuous by \cite{kohlenbach2}*{Prop.\ 3.12}.
In case $(\exists^{2})$, fix $x\in \R, f:\R\di \R$ and suppose $f$ is net-continuous at $x$, but not epsilon-delta continuous at $x$, i.e.\ there is $\eps_{0}>0$ such that
\be\label{barf}\textstyle
(\forall k\in \N)(\exists y\in \R)(|x-y|<_{\R}\frac{1}{2^{k}}\wedge |f(x)-f(y)|\geq_{\R} \eps_{0}).
\ee
Using $(\exists^{2})$, let $D$ be the set of all $y\in \R$ such that $|f(x)-f(y)|\geq_{\R} \eps_{0}$ and define `$y_{1}\preceq y_{2}$' for $y_{1}, y_{2}\in D$ by $|x-y_{1}|\geq_{\R}|x-y_{2}|$.
Clearly, the relation $\preceq$ yields a directed set.
%is transitive, and \eqref{barf} guarantees that the second item of Definition \ref{nets} also holds.  
Now define a net $x_{d}:D\di \R$ by $x_{d}:=d$ and note that $x_{d}$ converges to $x$ by \eqref{barf}. 
By the net-continuity of $f$, $f(x_{d})$ then converges to $f(x)$, which yields a clear contradiction. 
\end{proof}
The previous proof highlights a conceptual advantage of nets compared to sequences: to define a sequence $\lambda n^{0}.x_{n}$, one has to list the members one by one.  In this light, to get a sequence from \eqref{barf}, $\QFAC^{0,1}$ seems unavoidable.
By contrast, to define a net $x_{d}$, one only needs to satisfy Definition \ref{nets}, i.e.\ show that there always \emph{exist} `bigger' (in the sense of $\preceq$) elements in the net \emph{without} listing them.  

\smallskip

Now, a \emph{modulus-of-continuity functional} computes a modulus of continuity for functionals in a certain class.  
Various results exist on the minimal complexity of the former (see e.g.\ \cite{beeson1, troelstra1, exu}).  Theorem \ref{zienoudaar} 
implies that a \emph{modulus-of-net-continuity functional} is readily computed from a \emph{modulus-of-continuity functional} (in $\RCAo$).
The former takes as input $f$ and a modulus of convergence for $\lim_{d}x_{d}=x$ (and also $x_{d}$ and $x$), and outputs a modulus of convergence for $\lim_{d}f(x_{d})=f(x)$.

\smallskip

The following corollary is similar to Theorem \ref{himenez}, as the `strong' version of the Lindel\"of lemma implies $\QFAC^{0,1}$ by \cite{dagsamV}*{\S5}. 
\begin{cor}
The system $\ZF$ cannot prove the local equivalence between net-continuity and sequential continuity on $\R$.
\end{cor}
In conclusion, nets have the advantage that the associated notion of net-continuity is locally equivalent to the usual epsilon-delta definition \emph{without} the use of the Axiom of Choice as in $\QFAC^{0,1}$. 

\subsection{Nets and closed sets}\label{neclo}
The results in the previous section are not the only example of nets obviating the need for the Axiom of Choice.  
Indeed, we discuss another example involving closed sets, and the notion of `sequential space' in particular.    
These results are of historical interest, as Engelking writes in \cite{engelkoning}*{p.\ 55}:
\begin{quote}
Sequential spaces and Fr\'echet spaces belonged to the folklore almost since the
origin of general topology, [\dots] %but they were first thouroughly examined by Franklin [in \cites{uretafranklin,uretafranklin2}].
\end{quote}
We now introduce our notion of open and closed set in Definition \ref{gonghu}.
As to compatibility with classical RM, note that if $Y:\R\di \R$ is continuous, it represents an open set for which `$x\in Y$' has the same complexity (with parameters) as a code for an open set in RM (see \cite{simpson2}*{II.5.6}).  
Also note that the notion `sequentially closed' is similar to that of `separably closed' (see e.g.\ \cite{withgusto}). 
\bdefi[Open and closed sets]\label{gonghu}
\begin{enumerate}
 \renewcommand{\theenumi}{\alph{enumi}}
\item We let $Y: \R \di \R$ represent subsets of $\R$ by writing `$x \in Y$' for `$|Y(x)|>_{\R}0$'.  
\item We call $Y$ ‘open' if for $x \in Y$, there is an open ball $B(x, r) \subset Y$ with $r^{0}>0$.  
\item We define `$Y^{\c}$' as the complement of $Y$, i.e.\ $x\in Y^{\c}\asa \neg(x\in Y)$.  
\item We call a set $F$ `closed' if its complement $F^{\c}$ is open.
\item We call a set $F$ `sequentially closed' if for any sequence $x_{n}$ and $x$ in $\R$, we have $[(\forall n\in \N)(x_{n}\in F) \wedge \lim_{n\di \infty}x_{n}= x]\di  x\in F$.
\item A space is \emph{sequential} if `sequentially closed' and `closed' coincide for subsets.  % for all subsets (see \cite{uretafranklin}*{\S1}).  
\end{enumerate}
\edefi
Trivially, a closed set in $\R$ is sequentially closed, but the reverse direction cannot be proved in $\ZF$ by \cite{heerlijkheid}*{p.\ 73}.  
We prove that $\QFAC^{0,1}$ suffices over $\RCAo$.
\begin{thm}\label{XxX}
The system $\RCAo+\QFAC^{0,1}$ proves that $\R$ is a sequential space.  
\end{thm}
\begin{proof}
We prove the theorem in case $(\exists^{2})$ and in case $\neg(\exists^{2})$, and let the law of excluded middle finish the proof.
For the first case, let $Y$ be a sequentially closed sub-set of $\R$ and suppose that $Y$ is not closed, i.e.\ there is $x\in Y^{\c}$ such that 
\be\label{beehive}\textstyle
(\forall n\in \N)(\exists y\in \R)\big[|x-y|<\frac{1}{2^{n}}\wedge y\not\in Y^{\c}\big].
\ee
The formula in square brackets is arithmetical, and $\QFAC^{0,1}$ and $(\exists^{2})$ yield a sequence $y_{n}^{0\di 1}$ in $[0,1]$ such that $(\forall n\in \N)(|x-y_{n}|<\frac{1}{2^{n}}\wedge y_{n}\in Y)$.   
Clearly, $y_{n}$ converges to $x$, implying that $x\in Y$, a contradiction. 
In case $\neg(\exists)$, all $\R\di \R-$functions are continuous by \cite{kohlenbach2}*{Prop.\ 3.7}.  Hence, \eqref{beehive} immediately implies:
\[\textstyle
(\forall n\in \N)(\exists q\in [0,1]\cap \Q)\big[|x-q|<\frac{1}{2^{n}}\wedge q\in Y\big], 
\]
and $\QFAC^{0,0}$ now provides the required sequence in $Y$ as in the previous case. 
\end{proof}
\noindent
%The similarly between \eqref{barf} and \eqref{beehive}, and the associated proofs, is clear.  
%
%\smallskip
%\noindent
We call a set $F$ `net-closed' if for any net $x_{d}:D\di \R$ and $x\in \R$, we have that:
\be\textstyle\label{oppur}
[(\forall d\in D)(x_{d}\in F) \wedge \lim_{d}x_{d}= x]\di  x\in F.
\ee
Note that the definition of closed set in domain theory (\cite{gieren}*{p.\ 45}) \emph{is} that of net-closed.  
It should be noted that \emph{Scott open/Scott closed} is a much more important/central notion than net-open/net-closed in domain theory.   
As in Section \ref{NAP2}, the upgrade to nets obviates the need for $\QFAC^{0,1}$.
\begin{thm}\label{nienomaar}
The system $\RCAo$ proves that any net-closed set in $\R$ is closed.  
\end{thm}
\begin{proof}
We prove the theorem in case $(\exists^{2})$ and in case $\neg(\exists^{2})$, and let the law of excluded middle finish the proof.
In case $(\exists^{2})$, fix $F:\R\di \R$ and suppose $F$ is net-closed and not closed, i.e.\ there is $x\in F^{\c}$  such that
\be\label{barf2}\textstyle
(\forall k\in \N)(\exists y\in \R)(|x-y|<_{\R}\frac{1}{2^{k}}\wedge y\not \in F^{\c}).
\ee
Using $(\exists^{2})$, let $D$ be the set of all $y\in \R$ such that $y\in F$ (which is exactly `$y\not\in F^{\c}$') and define `$y_{1}\preceq y_{2}$' for $y_{1}, y_{2}\in D$ by $|x-y_{1}|\geq_{\R}|x-y_{2}|$.
Clearly, the relation $\preceq$ yields a directed set.  
%transitive, and \eqref{barf2} guarantees that the second item of Definition \ref{nets} also holds.  
Now define a net $x_{d}:D\di \R$ by $x_{d}:=d$ and note that $x_{d}$ converges to $x$ by \eqref{barf2}. 
By \eqref{oppur} and $(\forall d\in D)(x_{d}\in D)$, we have $x\in F$, a contradiction.  Hence, $F$ is closed and this case is finished.  The remaining case is treated as in the proof of Theorem~\ref{oppur}, i.e.\ using $\QFAC^{0,0}$.   
\end{proof}
Intuitively, a space is sequential if the topology can be described using sequences only, i.e.\ nets are not needed (see \cite{engelkoning}*{p.\ 53}).  Since all first-countable spaces are sequential (\cite{engelkoning}*{1.6.14}), the latter property is fairly weak.  
It is therefore somewhat ironic that $\QFAC^{0,1}$ is required to prove that $\R$ is sequential, while the base theory can establish this result for sequences replaced by nets.  Due to the classical equivalence, the same holds for the anti-Specker property from Section \ref{conets}. 

\smallskip

Finally, other results can be obtained in the same way: on one hand, $\QFAC^{0,1}$ is needed to show that every accumulation point of a set in $\R$ has a sequence converging to that point (\cite{heerlijkheid}*{p.\ 73}). 
On the other hand, $\RCAo$ can prove that every accumulation point of a set in $\R$ has a \emph{net} converging to that point.  

\subsection{Nets and sub-continuity}\label{pitche}
As suggested by its name, \emph{sub-continuity} is a notion of continuity (based on nets) that is strictly weaker than continuity.  Sub-continuity was introduced in \cite{voller} as in Definition \ref{aka} below.  
%Pincherle's theorem (see \cite{tepelpinch}*{p.\ 67}) states that a locally bounded function is bounded on certain compact domains and was studied in \cite{dagsamV} in RM and computability theory.  
%We study the interplay between sub-continuity, local boundedness, Pincherele's theorem, and countable choice.  
%In particular, we observe that nets obviate the need for strong comprehension \emph{and} countable choice in the proof of Pincherele's theorem.
Now, in \cite{dagsamV}*{\S4.2}, it is shown that sub-continuity \emph{involving sequences}, as found in e.g.\ \cite{noiri}, implies local boundedness using $\QFAC^{0,1}$.
We believe the use of countable choice to be necessary \emph{in the case of sequences}; we show in Theorem \ref{forgu} that sub-continuity \emph{formulated with nets} implies local boundedness over $\RCAo$.
%We sketch an equivalent version of Pincherle's theorem based on an existing notion of continuity, called \emph{subcontinuity}.  
%As it happens, subcontinuity is actually used in (applied) mathematics in various contexts: see e.g.\ \cite{gordon3}*{\S4.7}, \cite{migda}*{\S14.2}, \cite{mizera}*{p.\ 318}, and \cite{lola}*{\S4}.  
%
%\smallskip
%
%First of all, in a rather general setting, local boundedness is equivalent to the notion of \emph{subcontinuity}, introduced by Fuller in \cite{voller}.  
%The equivalence between subcontinuity and local boundedness (for first-countable Haussdorf spaces $X$ and functions $f:X\di \R$) may be found in \cite{roykes}*{p.\ 252}.  For the purposes of this paper, we restrict ourselves to $I\equiv [0,1]$, which simplifies the definition.  
\bdefi[Sub-continuity]\label{aka}
A function $f:\R\di \R$ is \emph{sub-continuous} if for any net $x_{d}:D\di I$ convergent to $ x\in \R$, $f(x_{d})$ has a convergent sub-net.  
\edefi
Note that $\lim_{d}f(x_{d})$ need not be $f(x)$ in the previous definition. 
Recall that a function is locally bounded if for every point there is a neighbourhood in which the functions is bounded.  
%Secondly, the equivalence between sub-continuity and local boundedness (without realisers) can then be proved as in Theorem \ref{forgu}.  The weak base theory in the latter constitutes a surprise: subcontinuity 
%has a typical `sequential compactness' flavour, while local boundedness has a typical `open-cover compactness' flavour.  The former and the latter are classified in the RM of resp.\ $\ACA_{0}$ and $\WKL$ ($\HBU$).      
\begin{thm}\label{forgu}
The system $\RCAo$ proves that a function $f:\R\di \R$ is locally bounded if it is sub-continuous. 
\end{thm}
\begin{proof}
We establish the theorem in $\RCAo$ in two steps: first we prove it assuming $(\exists^{2})$ and then prove it again assuming $\neg(\exists^{2})$.  
The law of excluded middle as in $(\exists^{2})\vee \neg(\exists^{2})$ then yields the theorem.  
Hence, assume $(\exists^{2})$ and suppose $f:\R\di \R$ is sub-continuous on $I$ but not locally bounded.  The latter assumption implies that there is $y_{0}\in \R$ such that 
\be\label{contrake}\textstyle
{(\forall n^{0})(\exists x\in \R)(|x-y_{0}|<_{\R}\frac{1}{n+1}\wedge |f(x)|>_{\R}n)}.  
\ee
Using $(\exists^{2})$, let $D$ be the set of all pairs $x\in \R$ and $n\in \N$ such that $0<|x-y_{0}|<_{\R}\frac{1}{n+1}\wedge |f(x)|>_{\R}n$.
Also define `$(y_{1}, n_{1})\preceq_{D} (y_{2}, n_{2})$' by $n_{1}\leq_{0} n_{2}$ for elements of $D$.
%$y_{1}, y_{2}\in D$ by $|y_{0}-y_{1}|\geq_{\R}|y_{0}-y_{2}|$.
Clearly, the relation $\preceq_{D}$ yields a directed set.  
%transitive, and \eqref{barf2} guarantees that the second item of Definition \ref{nets} also holds.  
Now define a net $x_{d}:D\di \R$ by $x_{d}:=y$ if $d=(y,n)$ and note that $\lim_{d}x_{d}=y_{0}$ by \eqref{contrake}. 
Hence the net $f(x_{d})$ has a convergent sub-net by sub-continuity, which is impossible as $f(x_{d})$ grows arbitrarily large by definition: $|f(x_{d})|>n$ if $d=(y,n)$ in particular. 
%Both conjuncts in \eqref{contrake} are $\Sigma_{1}^{0}$-formula, i.e.\ we may apply $\QFAC^{0,1}$ to \eqref{contrake} to obtain $\Phi^{0\di 1}$ such that for $y_{n}:=\Phi(n)$ and $x_{0}$ as in \eqref{contrake}, we have 
%\be\label{missyoumuch}\textstyle
%(\forall n\in \N)(|y_{n}-x_{0}|<_{\R}\frac{1}{n+1}\wedge |f(y_{n})|>_{\R}n), 
%\ee
%Clearly $y_{n}$ converges to $x_{0}$, and hence for some function $g:\N\di \N$, the subsequence $f(y_{g(n)})$ converges to some $y\in \R$ by the subcontinuity of $f$.  
%However, $f(y_{g(n)})$ also grows arbitrarily large by \eqref{missyoumuch}, a contradiction, and the reverse implication follows.  
%
%\smallskip
%
%Next, again assume $(\exists^{2})$; for the forward implication, suppose $f$ is locally bounded and let $x_{d}$ be a net in $I$ convergent to $x_{0}\in I$.  
%Then there is $k\in \N$ such that for all $y\in B(x_{0}, \frac{1}{k})$, $|f(y)|\leq k$.  However, for $n$ large enough, $x_{d}$ lies in $B(x, \frac{1}{k})$, implying that $|f(x_{d})|\leq k$ eventually.  
%In other words, the net $f(y_{d})$ eventually lies in the interval $[-k, k]$, and hence has a convergent sub-net by $(\exists^{2})$ and \cite{simpson2}*{I.9.3}.
%Thus, $f$ is subcontinuous, and we are done with the case $(\exists^{2})$.

\smallskip

Finally, in case that $\neg(\exists^{2})$, any function $f:\R\di \R$ is everywhere sequentially continuous and everywhere $\eps$-$\delta$-continuous by \cite{kohlenbach2}*{Prop.\ 3.12}.  
Hence, any $f:\R\di \R$ is also sub-continuous on $I$ and locally bounded on $I$, and the implication from the theorem is then trivially true.  
\end{proof}

\appendix
\section{General index sets}\label{naarhetgasthuis}
\subsection{Introduction}
The main part of this paper is devoted to the RM-study of nets indexed by subsets of Baire space. 
Our principal motivation for this restriction was simplicity: we already obtain $\HBU$ and $\FIVE$ from basic 
theorems pertaining to such nets (sometimes over $\ACAo$).  In this appendix, we show that nets become more powerful when
the index set is more general. 

\smallskip

In Section \ref{CTH}, we show that for index sets expressible in $\L_{n}$ ($n\geq 2$), the language of $n$-th order
arithmetic, we obtain full $n$-th order arithmetic from a realiser for the associated monotone convergence theorem for nets. 
Thus, the `size' of a net is directly proportional to the power of the associated convergence theorem. 

\smallskip

In Section \ref{ARM}, we study the sequentialisation principle $\SUB_{0}$ for larger index sets.  
In particular, we obtain an equivalence involving this principle for nets indexed by subsets of $\N^{\N}\di \N$ and $\QFAC^{0,2}$.  
The general case involving $\QFAC^{0,\sigma}$ is immediate.  Thus, the `size' of a net is directly proportional to the power of the associated sequentialisation theorem. 

%we obtain $\Pi_{1}^{n+1}$-comprehension from $n$-th arithmetic plus the Bolzano-Weierstrass theorem for nets in Cantor space with index sets expressible in $\L_{n}$ ($n\geq 2$). 
%% the language of $n$-th order
%%arithmetic. 
%Again, the `size' of a net is directly proportional to the power of the associated convergence theorem. 

\smallskip

 We stress that the results in this Appendix 
are included by way of illustration: the general study of nets is perhaps best undertaken in a suitable set theoretic framework.  That is not to say this section should be dismissed as \emph{spielerei}: 
index sets beyond Baire space do occur `in the wild', namely in \emph{fuzzy mathematics} and the \emph{iterated limit theorem}, as discussed in Remark \ref{fuzzytop} next. 
\begin{rem}[Large index sets]\label{fuzzytop}\rm
Zadeh founded the field of \emph{fuzzy mathematics} in \cite{zadeh65}.  
The core notion of \emph{fuzzy set} is a mapping that assigns values in $[0,1]$, i.e.\ a `level' of membership, rather than the binary relation from usual set theory.  
The first two chapters of Kelley's \emph{General Topology} (\cite{ooskelly}) are generalised to the setting of fuzzy mathematics in \cite{pupu}.  
As an example, \cite{pupu}*{Theorem 11.1} is the fuzzy generalisation of the classical statement that a point is in the closure of a set if and only if there is a net that converges to this point.  
However, as is clear from the proof of this theorem, to accommodate fuzzy points in $X$, the net is indexed by the space $X\di [0,1]$.  Moreover, the \emph{iterated limit theorem} (both the fuzzy and classical versions: \cite{pupu}*{Theorem 12.2} and \cite{ooskelly}) involves 
an index set $E_{m}$ \emph{indexed by $m\in D$}, where $D$ is an index set.  Thus, `large' index sets are found in the wild.       
\end{rem}
%BABA
In conclusion, we may also formulate two arguments in favour of `large' index sets based on the above results and in \cite{dagsamIII}, as follows.

\smallskip

First of all, by way of an exercise, the reader should generalise the well-known formulation of the Riemann integral in terms of nets (see e.g.\ \cite{ooskelly}*{p.\ 79}) to the gauge integral, as studied in \cite{dagsamIII}*{\S3.3}. 
As will become clear, this generalisation involves nets indexed by $\R\di \R$-functions.

\smallskip

Secondly, the results in Sections \ref{NAP2}-\ref{pitche} connect continuity and open sets to nets, all in $\R$.  As is clear from the proofs (esp.\ the use of the net $x_{d}:=d$ or similar), replacing $\R$ by a larger space requires the 
introduction of nets with a similarly large index set.  In particular, to show that a net-closed set $C$ is closed (see Theorem~\ref{nienomaar} for $C\subseteq \R$), one needs nets with an index set the same cardinality as $C$.

\subsection{Computability theory}\label{CTH}
We study the computational power of realiser for the monotone convergence theorem for nets indexed by `large' index sets.
To this end, we introduce the following hierarchy of comprehension functionals:
\be\tag{$\exists^{\sigma +2}$}
(\exists E^{(\sigma\di 0)\di 0})(\forall Y^{\sigma\di 0})\big[  E(Y)=_{0}0\asa (\exists f^{\sigma})(Y(f)=0)   \big].  
\ee
where $\sigma$ is any finite type.  Similar to Definition \ref{strijker}, we introduce the following.  
\bdefi[$\RCAo$]\label{strijker2}
A `subset $E$ of $\N^{\N}\di \N$' is given by its characteristic function $F_{E}^{3}\leq_{3}1$, i.e.\ we write `$Y\in E$' for $ F_{E}(Y)=1$ for any $Y^{2}$.
A `binary relation $\preceq$ on the subset $E$ of $ \N^{\N}\di \N$' is given by the associated characteristic function $G_{\preceq}^{(2\times 2)\di 0}$, i.e.\ we write `$Y\preceq Z$' for $G_{\preceq}(Y, Z)=1$ and any $Y, Z\in E$.
%Assuming extensionality on the reals as in item \eqref{EXTEN}, we obtain characteristic functions that represent subsets of $\R$ and relations thereon.  
%Using pairing functions, it is clear we can also represent sets of finite sequences (of reals), and relations thereon.  
\edefi
Secondly, let $\MCT_{\net}^{1}$ be the statement that any increasing net $x_{e}:E\di [0,1]$, i.e.\ indexed by subsets of $\N^{\N}\di \N$, converges to a limit in $[0,1]$.
A realiser for $\MCT_{\net}^{1}$ is a fifth-order object that takes as input $(E, \preceq_{E}, x_{e})$ and outputs the real $x=_{\R}\lim_{e} x_{e}$ if the inputs satisfy the conditions of the theorem. 
Similar to Corollary~\ref{koon}, we have the following elegant result. 
\begin{thm}\label{koon2}
A realiser for $\MCT_{\net}^{1}$ computes $\exists^{4}$ via a term of G\"odel's $T$, and vice versa.
\end{thm}
\begin{proof}
For the `vice versa' direction, one uses the usual `interval halving technique' where $\exists^{4}$ is used to decide whether there is $e\in E$ such that $x_{e}$ is in the relevant interval.  For the other direction, fix $F^{3}$, let $E$ be $\N^{\N}\di \N$ itself, and define `$X\preceq Y $' by $F(X)\geq_{0} F(Y)$ for any $X^{2}, Y^{2}$.     
It is straightforward to show that $(E, \preceq)$ is a directed set.  
Define the net $x_{e}:E \di I$ by $0$ if $F(e)>0$, and $1$ if $F(e)=0$, which is increasing by definition.  
Hence, $x_{e}$ converges, say to $y_{0}\in I$, and if $y_{0}>2/3$, then there must be $Y^{2}$ such that $F(Y)=0$, while if $y_{0}<1/3$, then $(\forall Y^{2})(F(Y)>0)$.   
Clearly, this yields a term of G\"odel's $T$ that computes $\exists^{3}$.  
\end{proof}
Let $\MCT_{\net}^{\sigma}$ be the obvious generalisation of $\MCT_{\net}^{1}$ to sets of type $\sigma+1$ objects.  
A realiser for the former computes $\exists^{\sigma+3}$, and vice versa, via a straightforward modification of Theorem \ref{koon2}.  Hence, the general monotone convergence theorem for nets is extremely hard to prove, even compared to e.g. $\exists^{3}$. 

\smallskip

Thirdly, we also study a special case of $\MCT_{\net}^{0}$ as follows.  
Let $\MCT_{\net}^{\SS}$ be $\MCT_{\net}^{0}$ restricted to directed sets $(D, \preceq)$ and nets $x_{d}:D\di I$ \emph{defined via arithmetical formulas}.  To be absolutely clear, we assume that `arithmetical formulas' are part of $\L_{2}$, i.e.\ \emph{only type zero and one parameters} are allowed.
%Next, it is well-known that $\exists^{2}$ computes a realiser for the monotone convergence theorem for \emph{sequences} via a term of G\"odel's $T$, 
%and vice versa (see \cite{yamayamaharehare}*{\S4}).  
%Inspired by this observation, we obtain an elegant `one type up' generalisation in Corollary~\ref{koon}.        
%A realiser for $\MCT_{\net}^{0}$ is a type $3$ functional taking as input $(D, \preceq_{D}, x_{d})$ and outputting the real $x=_{\R}\lim_{d} x_{d}$ if the inputs satisfy the conditions of $\MCT_{\net}^{0}$. 
\begin{thm}\label{koon3}
A realiser for $\MCT_{\net}^{\SS}$ computes $\SS^{2}$ via a term of G\"odel's $T$, and vice versa.
\end{thm}
\begin{proof}
For the `vice versa' direction, one uses the usual `interval halving technique' where $\SS^{2}$ is used to decide whether there is $d\in D$ such that $x_{d}$ is in the relevant interval.  For the other direction, fix $f^{1}$, let $D$ be Baire space, and define `$h\preceq g$' by the following arithmetical formula
\[
(\forall n\in \N)(\exists m\in \N)\big[  f(\overline{g}n)>0 \di f(\overline{h}m)\geq f(\overline{g}n)  \big],
\]
for any $h, g\in D$.     
It is straightforward to show that $(D, \preceq)$ is a directed set.  
Define the net $x_{g}:D \di I$ by $0$ if $(\exists n^{0})(f(\overline{g}n)>0)$, and $1$ if otherwise, which is arithmetical and increasing.  
Hence, $x_{d}$ converges, say to $y_{0}\in I$, and if $y_{0}>2/3$, then there must be $g^{1}$ such that $(\forall n^{0})(f(\overline{g}n)=0)$, while if $y_{0}<1/3$, then $(\forall g^{1})(\exists n^{0})(f(\overline{g}n)>0)$.   
Clearly, this provides a term of G\"odel's $T$ that computes $\SS^{2}$.  
\end{proof}
The restriction on parameters in $\MCT_{\net}^{\SS}$ turns out to be essential: we show that allowing type two parameters yields Gandy's \emph{superjump}.
The latter corresponds to the Halting problem for computability on type two inputs.  Indeed, the superjump $\SJ^{3}$ was introduced in \cite{supergandy} by Gandy (essentially) as follows:
\be\tag{$\SJ^{3}$}
\SJ(F^{2},e^{0}):=
\begin{cases}
0 & \textup{ if $\{e\}(F)$ terminates}\\
1 & \textup{otherwise}
\end{cases},
\ee
where the formula `$\{e\}(F)$ terminates' is a $\Pi_{1}^{1}$-formula, defined by Kleene's S1-S9 and (obviously) involving type two parameters.  
Let $\MCT_{\net}^{\SJ}$ be $\MCT_{\net}^{0}$ restricted to directed sets $(D, \preceq)$ and nets $x_{d}:D\di I$ {defined via arithmetical formulas, possibly involving type two parameters}.
\begin{cor}\label{kooncor3}
A realiser for $\MCT_{\net}^{\SJ}$ computes $\SJ^{3}$ via a term of G\"odel's $T$.
\end{cor}
\begin{proof}
Let $(\forall f^{1})\varphi(f, F^{2}, e^{0})$ be the formula expressing that the $e$-th algorithm with input $F^{2}$ terminates, i.e.\ $\varphi(f, F, e)$ is arithmetical with type two parameters.  
Let $D$ be Baire space and define `$f\preceq_{D} g$' by $\varphi(f, F,e )\di \varphi(g, F, e)$, which readily yields a directed set.  The net $x_{d}:D\di \R$ is defined as follows:
$x_{f}$ is $0$ if $\varphi(f, F, e)$, and $1$ otherwise.   This net is increasing and $\MCT_{\net}^{\SJ}$ yields a limit $y_{0}\in I$; 
if $y_{0}>1/3$, then $\{e\}(F)$ does not terminate, and if $y<2/3$, then $\{e\}(F)$ terminates. 
\end{proof}
To obtain a realiser for $\ATR_{0}$ (only), one could formulate a version of $\MCT_{\net}^{0}$ restricted to directed sets $(D, \preceq)$ and nets $x_{d}:D\di I$ \emph{defined via a quantifier-free formula with \textbf{continuous} type two parameters}.  The technical details are however somewhat involved, and we omit the proof.  

\subsection{Reverse Mathematics}\label{ARM}
%BABA add  a continuity theorem (nets equivalent to eps-delta) for index sets in  $\N^{\N}\di \N$; use \kappa_{0} and the usual trick
We study $\SUB_{0}$ from Section \ref{NAP} generalised to nets indexed by subsets of $\N^{\N}\di \N$.  
We establish an equivalence involving $\QFAC^{0,2}$, and the general case involving $\QFAC^{0,\sigma}$ readily follows. 

\smallskip

First of all, we define the sequentialisation principle $\SUB_{1}$, where $E$ is any subset of $\N^{\N}\di \N$.
Thus, the principle $\SUB_{1}$ deals with fourth-order arithmetic. 
\bdefi[$\SUB_{1}$]
For $x_{e}:E\di I$ an increasing net converging to $x\in I$, there is $\Phi:\N\di E$ such that $\lambda n.x_{\Phi(n)}$ is increasing and $\lim_{n\di \infty}x_{\Phi(n)}=_{\R}x$.  
%If the net converges pointwise to a function g in C(X), there exists an increasing (respectively, decreasing) sequence of functions (fn) belonging to the net which converges pointwise to g.
\edefi
Recall that $\IND$ is the induction schema for all formulas of $\L_{\omega}$.  
%The system $\RCAo+\IND$ has the same first-order strength as $\ACA_{0}$.  
\begin{thm}\label{himenez13372}
The system $\ACAo+\IND$ proves $\SUB_{1} \di\QFAC^{0,2}$.
\end{thm}
\begin{proof}
%In case $\neg(\exists^{2})$, all functions on Baire space are continuous by \cite{kohlenbach2}*{Prop.~3.7}, and $\QFAC^{0,1}$ clearly reduces to $\QFAC^{0,0}$, included in $\RCAo$.  For the case $(\exists^{2})$, note that we also have $(\mu^{2})$.  
First of all, the first part of the proof of Theorem \ref{himenez1337} is dedicated to coding: namely to showing that if $(\forall n^{0})(\exists f^{1})(Y( n,f)=0)$, then there is $G^{2}$ with $(\forall n^{0})(\exists f^{1})(G(n,f)=0)$ and $(\forall f^{1}, n^{0}, m^{0})((G(n,f)=0\wedge m\leq n)\di G(m,f)=0)$.
This step is routine based on $\IND$ and we will just assume that $Y^{3}$ satisfies $(\forall n^{0})(\exists F^{2})(Y(n, F)=0)$ and $(\forall F^{2}, n^{0}, m^{0})((Y(n, F)=0 \wedge m\leq n)\di Y(m, F)=0)$.  The underlined formula in \eqref{zosimpel2} has this property anyway. 
%Let $\b^{1\di 1^{*}}$ be the inverse of a pairing function defined as $|\b(f)|=f(0)+1$ and $\b(f)(i)$ for $i<|\b(f)|$ is the sequence $f(1+i), f(1+i+|\b(f)|), f(1+i+2|\b(f)|), \dots$, which is definable in $\RCAo$.
%Fix some $F^{(0\times 1)\di 0}$ satisfying the antecedent of $\QFAC^{0,1}$, i.e.\ $(\forall n^{0})(\exists f^{1})(F(n, f)=0)$, and use $\IND$ to prove $(\forall n^{0})(\exists f^{1})\underline{(\forall i\leq n)(F(i, \b(\langle n \rangle * f)(i))=0)}$.  
%The underlined formula is also written `$G(n, f)=0$' and 

\smallskip

Secondly, if there is $F_{0}^{2}$ such that $(\forall n^{0})(Y(n,F_{0})=0)$, then the consequent of $\QFAC^{0,2}$ trivially holds.

\smallskip

Thirdly, in case $(\forall F^{2})(\exists n^{0})(Y(n,F)\ne 0)$, define the set $E:=\{F^{2}:(\exists n^{0} )Y(n, F)=0  \}$ and define the predicate `$\preceq_{E}$' as: $F\preceq_{E} G$ if and only if 
\be\label{cruzie2}
(\mu n)(Y(n, F)\ne 0)\leq (\mu m)(Y(m, G)\ne 0),
%(\mu n)(\forall i\leq n)(F(i, f\b(f)(i)=0)\leq (\mu m)(\forall j\leq m)(F(j,\b(g)(j))=0).
\ee
which is well-defined by assumption.  
Note that $E$ with $\preceq_{E}$ forms a directed set by assumption. 
Define the increasing net $x_{e}:= 1-2^{-(\mu n)(Y(n, e)\ne0)}$ and note that $\lim_{e}x_{e}=1$ by assumption and \eqref{cruzie2}.  By $\SUB_{1}$, there is some $\Phi^{0\di 2}$ such that $\lim_{n\di \infty} x_{\Phi(n)}=1$, i.e.\ 
$(\forall \eps>0)(\exists m^{0})(\forall k^{0}\geq m)(|x_{\Phi(k)}-1|<\eps)$, and use $\mu^{2}$ to find $\Psi^{2}$ computing such $m^{0}$ from $\eps$.  
Then the functional $Z(n):= {\Phi(\Psi(\frac{1}{2^{n+1}}))}$ provides the witness as required for the conclusion of $\QFAC^{0,2}$.  
\end{proof} 
%Let $\ADS$ be the $\L_{2}$-sentence from the RM zoo (see \cite{dsliceke}*{Def.\ 9.1}) that every infinite linear order has an infinite ascending or descending sequence.
\begin{cor}\label{zienoudoor2}
The system $\ACAo+\IND+\ADS$ proves $\SUB_{1} \asa \QFAC^{0,2}$.
\end{cor}
\begin{proof}
We only need to prove the reverse implication.  To this end, let $x_{e}:E\di I$ be an increasing net converging to some $x\in I$.  This convergence trivially implies:
\be\label{zosimpel2}\textstyle
(\forall k\in \N)(\exists e\in E)\underline{(|x-x_{e}|<\frac{1}{2^{k}})}, 
\ee
and applying $\QFAC^{0, 2}$ to \eqref{zosimpel2} yields $\Phi:\N\di E$ such that \emph{the sequence} $\lambda k^{0}.x_{\Phi(k)}$ also converges to $x$ as $k\di \infty$.  
Since $\ADS$ is equivalent to the statement that every sequence in $\R$ has a monotone sub-sequence (see \cite{kruiserover}*{\S3}), $\SUB_{1}$ now follows. 
\end{proof}
Let $\SUB_{\sigma}$ be the obvious generalisation of $\SUB_{1}$ to sets of type $\sigma+1$ objects.  
A straightforward modification of the proof of Theorem \ref{himenez13372} and its corollary then yields $\QFAC^{0,\sigma+1}\asa \SUB_{\sigma}$.
Hence, a general sequentialisation theorem for nets is extremely hard to prove in that it would require full countable choice. 

\smallskip

Finally, we obtain a nice splitting for $\QFAC^{0,2}$ based on the following sequentialisation principle, where $E$ (resp.\ $D$) is any subset of $\N^{\N}\di \N$ (resp.\ $\N^{\N}$).
\bdefi[$\SUB_{\frac{1}{2}}$]
For $x_{e}:E\di I$ an increasing net converging to $x\in I$, there is $\Phi:D\di E$ such that the net $\lambda d.x_{\Phi(d)}$ is increasing and $\lim_{d}x_{\Phi(d)}=_{\R}x$.  
%If the net converges pointwise to a function g in C(X), there exists an increasing (respectively, decreasing) sequence of functions (fn) belonging to the net which converges pointwise to g.
\edefi
\begin{thm}
$\RCAo+\IND+\ADS$ proves $\QFAC^{0,2}\asa [\SUB_{\frac12}+\QFAC^{0,1}]$.
\end{thm}
\begin{proof}
Immediate from Corollaries \ref{zienoudoor} and \ref{zienoudoor2}.
\end{proof}
It goes without saying that the above results provides \emph{mutadis mutandis} a whole hierarchy involving $\QFAC^{0,\sigma}$ and the associated (obvious) generalisations of $\SUB_{\frac12}$.

\smallskip

Finally, we note that \cite{samph} already includes a natural equivalence involving $\QFAC^{0,2}$.

\section{Nets and the G\"odel hierarchy}\label{kodel}
We discuss the foundational implications of our results, esp.\ as they pertain to the \emph{G\"odel hierarchy}.
Now, the latter is a collection of logical systems ordered via consistency strength.  This hierarchy is claimed to capture most systems that are natural or have foundational import, as follows. 
\begin{quote}
\emph{It is striking that a great many foundational theories are linearly ordered by $<$. Of course it is possible to construct pairs of artificial theories which are incomparable under $<$. However, this is not the case for the ``natural'' or non-artificial theories which are usually regarded as significant in the foundations of mathematics.} (\cite{sigohi})
\end{quote}
Burgess and Koellner corroborate this claim in \cite{dontfixwhatistoobroken}*{\S1.5} and \cite{peterpeter}*{\S1.1}.
The G\"odel hierarchy is a central object of study in mathematical logic, as e.g.\ argued by Simpson in \cite{sigohi}*{p.\ 112} or Burgess in \cite{dontfixwhatistoobroken}*{p.\ 40}.  
Precursors to the G\"odel hierarchy may be found in the work of Wang (\cite{wangjoke}) and Bernays (see \cite{theotherguy,puben}).
Friedman (\cite{friedber}) studies the linear nature of the G\"odel hierarchy in detail.  % than present in Figure \ref{xxy} below.
Moreover, the G\"odel hierarchy exhibits some remarkable \emph{robustness}: we can perform the following modifications and the hierarchy remains largely unchanged:
\begin{enumerate}
 \renewcommand{\theenumi}{\roman{enumi}}
\item Instead of the consistency strength ordering, we can order via inclusion: Simpson claims that inclusion and consistency strength yield the same\footnote{Simpson mentions in \cite{sigohi} the caveat that e.g.\ $\PRA$ and $\WKL_{0}$ have the same first-order strength, but the latter is strictly stronger than the former.} G\"odel hierarchy as depicted in \cite{sigohi}*{Table 1}.  Some exceptional (semi-natural) statements\footnote{There are some examples (predating $\HBU$ and \cite{dagsamIII}) that fall outside of the G\"odel hierarchy \emph{based on inclusion}, like \emph{special cases} of Ramsey's theorem and the axiom of determinacy from set theory (\cites{dsliceke, shoma}).  These are far less natural than e.g.\ Heine-Borel compactness, in our opinion.} do fall outside of the inclusion-based G\"odel hierarchy.\label{kut} 
\item We can replace the systems with their higher-order (eponymous but for the `$\omega$') counterparts.  The higher-order systems are generally conservative over their second-order counterpart for (large parts of) $\L_{2}$.  Hunter's dissertation contains a number of such general results (\cite{hunterphd}*{Ch.\ 2}).\label{lul}
\end{enumerate}
Now, \emph{if} one accepts the modifications (inclusion ordering and higher types) described in the previous two items, \emph{then} an obvious question is where e.g.\ $\HBU$ fits into the (inclusion-based) G\"odel hierarchy.  Indeed, the Heine-Borel theorem has a central place in analysis and a rich history predating set theory (see \cite{medvet}).

\smallskip

The answer to this question may come as a surprise: starting with the results in \cite{dagsamIII,dagsamV,dagsamVI}, Dag Normann and the author have identified a \emph{large} number of \emph{natural} theorems of third-order arithmetic, including $\HBU$, forming a branch \emph{independent} of the medium range of the G\"odel hierarchy based on inclusion.  Indeed, none of the systems $\SIXK+\QFAC^{0,1}$ can prove $\HBU$, while $\Z_{2}^{\Omega}$ can.  We stress that both $\SIXK+\QFAC^{0,1}$ and $\HBU$ are part of the language of \emph{third-order arithmetic}, i.e.\ expressible in the same language. 

\smallskip

In more detail, results pertaining to `local-global' theorems are obtained in \cite{dagsamV}.  Measure theory is studied in \cite{dagsamVI}, while results pertaining to $\HBU$ and the gauge integral may be found in \cite{dagsamIII}.  
In this paper and \cite{samcie19,dagsamVI,samwollic19}, we have shown that a number of basic theorems about nets similarly fall outside of the G\"odel hierarchy including 
%the Bolzano-Weierstrass theorem for nets ($\BW_{\net}$; see \cite{samcie19}) and 
the monotone convergence theorem for nets of continuous functions and the Riemann integral ($\MCT_{\net}$; see \cite{dagsamVI}).

\smallskip

We recall that convergence theorems concerning nets are old and well-established, starting with Moore, Smith, and Vietoris more than a century ago \cite{moorelimit2,moorsmidje,kliet}.  
Our results highlight a fundamental difference between second-order and higher-order arithmetic.  
Such differences are discussed in detail in \cite{samsplit}*{\S4}, based on helpful discussion with Steve Simpson, Denis Hirschfeldt, and Anil Nerode. 
The associated results concerning nets are summarised in Figure \ref{xxy} below.      
\begin{figure}[h]
\[
\begin{array}{lll}
&\textup{\textbf{strong}} \hspace{1.5cm}& 
\left\{\begin{array}{l}
%\vdots\\
%\textup{supercompact cardinal}\\
%\vdots\\
%\textup{measurable cardinal}\\
\vdots\\
\ZFC \\
\textsf{\textup{ZC}} \\
\textup{simple type theory}
\end{array}\right.
\\
&&\\
  &&\quad{ {~\Z_{2}^{\Omega}}} \\
{ {\begin{array}{l}
~\\
\textup{Convergence theorems for }\\
\textup{nets: $\MCT_{\net}^{0}$, $\BW_{\net}$, $\AS_{\net}$}\\
%\textup{compactness and dimension}\\
%\textup{for the real numbers $\R$}
\end{array}}}
&\textup{\textbf{medium}} & 
\left\{\begin{array}{l}
 {\Z}_{2}^{\omega}+ \QFAC^{0,1}\\
\vdots\\
\textup{$\Pi_{2}^{1}\textsf{-CA}_{0}^{ {\omega}}$}\\
\textup{$\FIVE^{ {\omega}}$ }\\
\textup{$\ATR_{0}^{ {\omega}}$}  \\
\textup{$\ACA_{0}^{ {\omega}}$} \\
\end{array}\right.
%\begin{array}{c}
%\textup{Kohlenbach's}\\
%\textup{ {higher-order RM}}\\
%\end{array}
\\
~\\
{ {\left.\begin{array}{l}
%\textup{}\\
\textup{$\HBU$, Dini's theorem for}\\
\textup{nets, convergence theorem}\\
\textup{for nets and the Riemann}\\
\textup{integral: $\MCT_{\net}$}
\end{array}\right\}}}
&\begin{array}{c}\\\textup{\textbf{weak}}\\ \end{array}& 
\left\{\begin{array}{l}
\WKL_{0}^{ {\omega}} \\
\textup{$\RCA_{0}^{ {\omega}}$} \\
\textup{$\textsf{PRA}$} \\
\textup{$\textsf{EFA}$ } \\
\textup{bounded arithmetic} \\
\end{array}\right.
\\
\end{array}
\]
\caption{The G\"odel hierarchy with a side-branch for the medium range}\label{xxy}
\begin{picture}(250,0)
\put(36,120){ {\vector(0,-1){25}}}
%\put(26,167){ {\vector(0,-1){12}}}
%\put(195,220){ {\vector(-3,-1){140}}}
\put(185,195){ {\vector(-4,-1){110}}}
%\put(160,175){{\line(-5,3){20}}}

\multiput(120,70)(5,0){13}{\line(1,0){3}}
\put(180,70){ {\vector(1,0){3}}}

\put(85,138){ {\vector(4,-1){90}}}

\put(157,97){ {\vector(-3,-1){53}}}
\put(120,80){ {\vector(3,1){50}}}
\put(157,85){{\line(-5,3){20}}}
\end{picture}
\end{figure}
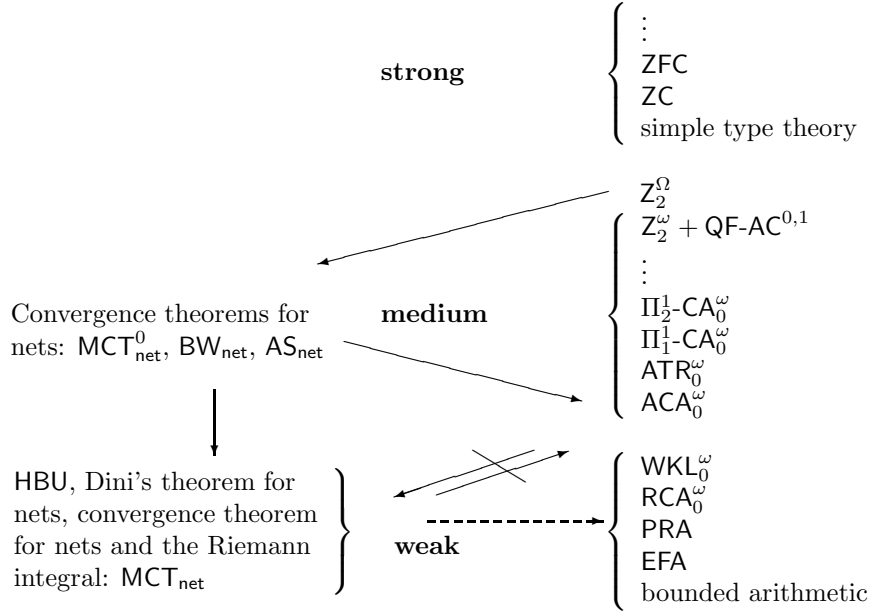~\\  
Finally, we discuss some the technical details concerning Figure \ref{xxy}.  % as follows.
\begin{rem}\rm\label{knellen}
First of all, $\Z_{2}^{\Omega}$ is placed \emph{between} the medium and strong range, as the combination of the recursor $\textsf{R}_{2}$ from G\"odel's $T$ and $\exists^{3}$ yields a system stronger than $\Z_{2}^{\Omega}$.  The system $\SIXK$ does not change in the same way.     

\smallskip

Secondly, while $\HBU$ clearly implies $\WKL$, $\MCT_{\net}$ from \cite{dagsamVI} only implies $\WWKL$ as far as we know, and this is symbolised by the dashed line.  % in Figure \ref{xxy}.  
\end{rem}
In conclusion, in light of the results in this paper and \cites{dagsamIII, dagsamV, dagsamVI, samcie19,samwollic19}, we observe a serious challenge to the linear nature of the G\"odel hierarchy (with a caveat provided by the above items \eqref{kut} and \eqref{lul}),
 as well as Feferman's claim that the mathematics necessary for the development of physics can be formalised in relatively weak logical systems (see e.g.\ \cite{dagsamIII}*{p.\ 24}).

\begin{ack}\rm
Our research was supported by the John Templeton Foundation via the grant \emph{a new dawn of intuitionism} with ID 60842.
We express our gratitude towards this institution. 
We thank Dag Normann, Thomas Streicher, and Anil Nerode for their valuable advice.  
%We also thank the anonymous referee for the helpful suggestions.  
Opinions expressed in this paper do not necessarily reflect those of the John Templeton Foundation.    
\end{ack}


\begin{bibdiv}
\begin{biblist}
%\bibselect{allkeida}
\bib{aju}{article}{
  author={Abramsky, Samson},
  author={Jung, Achim},
  title={Domain theory},
  conference={ title={Handbook of logic in computer science, Vol. 3}, },
  book={ series={Handb. Log. Comput. Sci.}, volume={3}, publisher={Oxford Univ. Press}, },
  date={1994},
  pages={1--168},
}

\bib{alavielvanzijnklapstoel}{article}{
  author={Alaoglu, L.},
  title={Weak topologies of normed linear spaces},
  journal={Ann. of Math. (2)},
  volume={41},
  date={1940},
  pages={252--267},
}

\bib{IDA}{book}{
  author={Aliprantis, Charalambos D.},
  author={Border, Kim C.},
  title={Infinite dimensional analysis},
  edition={3},
  note={A hitchhiker's guide},
  publisher={Springer, Berlin},
  date={2006},
  pages={xxii+703},
}

\bib{arse}{article}{
  author={Arzel\`a, Cesaro},
  title={Intorno alia continuitd delta somma d'infinitd di funzioni continue},
  year={1883--1884},
  journal={Rend. dell'Accad. di Bologna},
  pages={79--84},
}

\bib{avi2}{article}{
  author={Avigad, Jeremy},
  author={Feferman, Solomon},
  title={G\"odel's functional \(``Dialectica''\) interpretation},
  conference={ title={Handbook of proof theory}, },
  book={ series={Stud. Logic Found. Math.}, volume={137}, },
  date={1998},
  pages={337--405},
}

\bib{aviergo}{article}{
  author={Avigad, Jeremy},
  title={The metamathematics of ergodic theory},
  journal={Ann. Pure Appl. Logic},
  volume={157},
  date={2009},
  number={2-3},
  pages={64--76},
}

\bib{zonderfilter}{article}{
  author={Bartle, Robert G.},
  title={Nets and filters in topology},
  journal={Amer. Math. Monthly},
  volume={62},
  date={1955},
  pages={551--557},
}

\bib{bartlecomp}{article}{
  author={Bartle, Robert G.},
  title={On compactness in functional analysis},
  journal={Trans. Amer. Math. Soc.},
  volume={79},
  date={1955},
  pages={35--57},
}

\bib{zonbaardiemop}{book}{
  author={Beardon, Alan F.},
  title={Limits},
  series={Undergraduate Texts in Mathematics},
  note={A new approach to real analysis},
  publisher={Springer-Verlag, New York},
  date={1997},
  pages={x+189},
}

\bib{beeson1}{book}{
  author={Beeson, Michael J.},
  title={Foundations of constructive mathematics},
  series={Ergebnisse der Mathematik und ihrer Grenzgebiete},
  volume={6},
  note={Metamathematical studies},
  publisher={Springer},
  date={1985},
  pages={xxiii+466},
}

\bib{puben}{book}{
  author={Paul Benacerraf},
  author={Hilary Putnam},
  edition={2},
  title={Philosophy of Mathematics: Selected Readings},
  publisher={Cambridge University Press},
  year={1984},
}

\bib{bribe}{article}{
  author={Berger, Josef},
  author={Bridges, Douglas},
  title={A fan-theoretic equivalent of the antithesis of Specker's theorem},
  journal={Indag. Math. (N.S.)},
  volume={18},
  date={2007},
  number={2},
  pages={195--202},
}

\bib{diniberg}{article}{
  author={Berger, Josef},
  author={Schuster, Peter},
  title={Dini's theorem in the light of reverse mathematics},
  conference={ title={Logicism, intuitionism, and formalism}, },
  book={ series={Synth. Libr.}, volume={341}, publisher={Springer}, },
  date={2009},
  pages={153--166},
}

\bib{theotherguy}{article}{
  pages={52--69},
  year={1935},
  author={Paul Bernays},
  volume={34},
  title={Sur le Platonisme Dans les Math\'ematiques},
  journal={L'Enseignement Math\'ematique},
}

\bib{berkhofca}{article}{
  author={Birkhoff, Garrett},
  title={Abstract 355: A new definition of limit},
  journal={Bull. Amer. Math. Soc.},
  volume={41},
  date={1935},
  pages={p.\ 636},
}

\bib{berkhof}{article}{
  author={Birkhoff, Garrett},
  title={Moore-Smith convergence in general topology},
  journal={Ann. of Math. (2)},
  volume={38},
  date={1937},
  number={1},
  pages={39--56},
}

\bib{bish1}{book}{
  author={Bishop, Errett},
  title={Foundations of constructive analysis},
  publisher={McGraw-Hill},
  date={1967},
  pages={xiii+370},
}

\bib{brich}{book}{
  author={Bridges, Douglas},
  author={Richman, Fred},
  title={Varieties of constructive mathematics},
  series={London Mathematical Society Lecture Note Series},
  volume={97},
  publisher={Cambridge University Press},
  place={Cambridge},
  date={1987},
  pages={x+149},
}

\bib{bridentma}{article}{
  author={Bridges, Douglas},
  author={Dent, James},
  author={McKubre-Jordens, Maarten},
  title={Constructive connections between anti-Specker, positivity, and fan-theoretic properties},
  journal={New Zealand J. Math.},
  volume={44},
  date={2014},
}

\bib{opborrelen}{article}{
  author={Borel, Emile},
  title={Sur quelques points de la th\'eorie des fonctions},
  journal={Ann. Sci. \'Ecole Norm. Sup. (3)},
  volume={12},
  date={1895},
  pages={9--55},
}

\bib{gentoporg1}{book}{
  author={Bourbaki, N.},
  title={Elements of mathematics. General topology. Part 1},
  publisher={Addison-Wesley},
  date={1966},
  pages={vii+437},
}

\bib{gentoporg2}{book}{
  author={Bourbaki, Nicolas},
  title={Elements of mathematics. General topology. Part 2},
  publisher={Addison-Wesley},
  date={1966},
  pages={iv+363},
}

\bib{brace5imp}{article}{
  author={Brace, John W.},
  title={Almost uniform convergence},
  journal={Portugal. Math.},
  volume={14},
  date={1956},
  pages={99--104},
}

\bib{brace4imp}{article}{
  author={Brace, John W.},
  title={The topology of almost uniform convergence},
  journal={Pacific J. Math.},
  volume={9},
  date={1959},
  pages={643--652},
}

\bib{obro}{book}{
  author={Brown, Arlen},
  author={Pearcy, Carl},
  title={An introduction to analysis},
  series={Graduate Texts in Mathematics},
  volume={154},
  publisher={Springer},
  date={1995},
  pages={viii+297},
}

\bib{brownphd}{book}{
  author={Brown, Douglas K.},
  title={Functional analysis in weak subsystems of second-order arithmetic},
  year={1987},
  publisher={PhD Thesis, The Pennsylvania State University, ProQuest LLC},
}

\bib{browner}{article}{
  author={Brown, Douglas K.},
  title={Notions of compactness in weak subsystems of second order arithmetic},
  conference={ title={Reverse mathematics 2001}, },
  book={ series={Lect. Notes Log.}, volume={21}, publisher={Assoc. Symbol. Logic}, },
  date={2005},
  pages={47--66},
}

\bib{dontfixwhatistoobroken}{book}{
  author={Burgess, John P.},
  title={Fixing Frege},
  series={Princeton Monographs in Philosophy},
  publisher={Princeton University Press},
  date={2005},
  pages={x+257},
}

\bib{cartman}{article}{
  author={Cartan, H.},
  title={Th\'eorie des filtres \textup {and} Filtres et ultrafiltres},
  journal={C. R. Acad. Sc. Paris},
  volume={205},
  date={1937},
  pages={p.\ 595-598 \textup {and} p. 777-779},
}

\bib{cassacassa}{article}{
  author={Caserta, Agata},
  author={Di Maio, Giuseppe},
  author={Hol\'{a}, \v {L}ubica},
  title={Arzel\`a's theorem and strong uniform convergence on bornologies},
  journal={J. Math. Anal. Appl.},
  volume={371},
  date={2010},
  number={1},
  pages={384--392},
}

\bib{cousin1}{article}{
  author={Cousin, Pierre},
  title={Sur les fonctions de $n$ variables complexes},
  journal={Acta Math.},
  volume={19},
  date={1895},
  pages={1--61},
}

\bib{day}{article}{
  author={Day, Adam R.},
  title={On the strength of two recurrence theorems},
  journal={J. Symb. Log.},
  volume={81},
  date={2016},
  number={4},
  pages={1357--1374},
}

\bib{damirzoo}{misc}{
  author={Dzhafarov, Damir D.},
  title={Reverse Mathematics Zoo},
  note={\url {http://rmzoo.uconn.edu/}},
}

\bib{engelkoning}{book}{
  author={Engelking, Ryszard},
  title={General topology},
  series={Sigma Series in Pure Mathematics},
  volume={6},
  edition={2},
  publisher={Heldermann},
  date={1989},
  pages={viii+529},
}

\bib{exu}{article}{
  author={Escard{\'o}, Mart{\'{\i }}n},
  author={Xu, Chuangjie},
  title={The Inconsistency of a Brouwerian Continuity Principle with the Curry-Howard Interpretation},
  booktitle={13th International Conference on Typed Lambda Calculi and Applications (TLCA 2015)},
  pages={153--164},
  series={Leibniz International Proceedings in Informatics (LIPIcs)},
  year={2015},
  volume={38},
}

\bib{rimjob}{book}{
  author={Felgner, Ulrich},
  title={Models of $\textup {ZF}$-set theory},
  series={Lecture Notes in Mathematics, Vol. 223},
  publisher={Springer-Verlag, Berlin-New York},
  date={1971},
  pages={vi+173},
}

\bib{fried}{article}{
  author={Friedman, Harvey},
  title={Some systems of second order arithmetic and their use},
  conference={ title={Proceedings of the International Congress of Mathematicians (Vancouver, B.\ C., 1974), Vol.\ 1}, },
  book={ },
  date={1975},
  pages={235--242},
}

\bib{fried2}{article}{
  author={Friedman, Harvey},
  title={ Systems of second order arithmetic with restricted induction, I \& II (Abstracts) },
  journal={Journal of Symbolic Logic},
  volume={41},
  date={1976},
  pages={557--559},
}

\bib{friedber}{article}{
  author={Friedman, Harvey},
  title={Interpretations, According to Tarski},
  journal={Interpretations of Set Theory in Discrete Mathematics and Informal Thinking, The Nineteenth Annual Tarski Lectures, \url {http://u.osu.edu/friedman.8/files/2014/01/Tarski1052407-13do0b2.pdf}},
  date={2007},
  number={1},
  pages={pp.\ 42},
}

\bib{voller}{article}{
  author={Fuller, Richard V.},
  title={{Relations among continuous and various non-continuous functions.}},
  journal={{Pac. J. Math.}},
  volume={25},
  pages={495--509},
  year={1968},
}

\bib{furweiss}{article}{
  author={Furstenberg, H.},
  author={Weiss, B.},
  title={Topological dynamics and combinatorial number theory},
  journal={J. Analyse Math.},
  volume={34},
  date={1978},
  pages={61--85 (1979)},
}

\bib{supergandy}{article}{
  author={Gandy, Robin},
  title={General recursive functionals of finite type and hierarchies of functions},
  journal={Ann. Fac. Sci. Univ. Clermont-Ferrand No.},
  volume={35},
  date={1967},
  pages={5--24},
}

\bib{gerwei}{article}{
  author={Gerhardy, Philipp},
  title={Proof mining in topological dynamics},
  journal={Notre Dame J. Form. Log.},
  volume={49},
  date={2008},
  number={4},
  pages={431--446},
}

\bib{gieren2}{book}{
  author={Gierz, G.},
  author={Hofmann, K. H.},
  author={Keimel, K.},
  author={Lawson, J. D.},
  author={Mislove, M.},
  author={Scott, D. S.},
  title={A compendium of continuous lattices},
  publisher={Springer},
  date={1980},
  pages={xx+371},
}

\bib{gieren}{book}{
  author={Gierz, G.},
  author={Hofmann, K. H.},
  author={Keimel, K.},
  author={Lawson, J. D.},
  author={Mislove, M.},
  author={Scott, D. S.},
  title={Continuous lattices and domains},
  series={Encyclopedia of Mathematics and its Applications},
  volume={93},
  publisher={Cambridge University Press},
  date={2003},
  pages={xxxvi+591},
}

\bib{withgusto}{article}{
  author={Giusto, Mariagnese},
  author={Simpson, Stephen G.},
  title={Located sets and reverse mathematics},
  journal={J. Symbolic Logic},
  volume={65},
  date={2000},
  number={3},
  pages={1451--1480},
}

\bib{degou}{book}{
  author={Goubault-Larrecq, Jean},
  title={Non-Hausdorff topology and domain theory},
  series={New Mathematical Monographs},
  volume={22},
  publisher={Cambridge University Press},
  date={2013},
  pages={vi+491},
}

\bib{halmospalmos}{book}{
  author={Halmos, Paul R.},
  title={Introduction to Hilbert space and the theory of spectral multiplicity},
  note={Reprint of the second (1957) edition},
  publisher={AMS Chelsea Publishing},
  date={1998},
  pages={114},
}

\bib{andgrethel}{article}{
  author={Hansell, R. W.},
  title={Monotone subnets in partially ordered sets},
  journal={Proc. Amer. Math. Soc.},
  volume={18},
  date={1967},
  pages={854--858},
}

\bib{heerlijkheid}{book}{
  author={Herrlich, Horst},
  title={Axiom of choice},
  series={Lecture Notes in Mathematics},
  volume={1876},
  publisher={Springer},
  date={2006},
  pages={xiv+194},
}

\bib{dsliceke}{book}{
  author={Hirschfeldt, Denis R.},
  title={Slicing the truth},
  series={Lecture Notes Series, Institute for Mathematical Sciences, National University of Singapore},
  volume={28},
  publisher={World Scientific Publishing},
  date={2015},
  pages={xvi+214},
}

\bib{hunterphd}{book}{
  author={Hunter, James},
  title={Higher-order reverse topology},
  note={Thesis (Ph.D.)--The University of Wisconsin - Madison},
  publisher={ProQuest LLC, Ann Arbor, MI},
  date={2008},
  pages={97},
}

\bib{scrutihara}{article}{
  author={Ishihara, Hajime},
  author={Schuster, Peter},
  title={Compactness under constructive scrutiny},
  journal={MLQ Math. Log. Q.},
  volume={50},
  date={2004},
  number={6},
  pages={540--550},
}

\bib{ishi1}{article}{
  author={Ishihara, Hajime},
  title={Reverse mathematics in Bishop's constructive mathematics},
  year={2006},
  journal={Philosophia Scientiae (Cahier Sp\'ecial)},
  volume={6},
  pages={43-59},
}

\bib{ooskelly}{book}{
  author={Kelley, John L.},
  title={General topology},
  note={Reprint of the 1955 edition; Graduate Texts in Mathematics, No. 27},
  publisher={Springer-Verlag},
  date={1975},
  pages={xiv+298},
}

\bib{ooskelly2}{book}{
  author={Kelley, John L.},
  author={Srinivasan, T. P.},
  title={Measure and integral. Vol. 1},
  series={Graduate Texts in Mathematics},
  volume={116},
  publisher={Springer-Verlag, New York},
  date={1988},
  pages={x+150},
}

\bib{peterpeter}{incollection}{
  author={Koellner, Peter},
  title={Large Cardinals and Determinacy},
  booktitle={The Stanford Encyclopedia of Philosophy},
  editor={Edward N. Zalta},
  note={\url {https://plato.stanford.edu/archives/spr2014/entries/large-cardinals-determinacy/}},
  year={2014},
  edition={Spring 2014},
  publisher={Metaphysics Research Lab, Stanford University},
}

\bib{kohlenbach4}{article}{
  author={Kohlenbach, Ulrich},
  title={Foundational and mathematical uses of higher types},
  conference={ title={Reflections on the foundations of mathematics}, },
  book={ series={Lect. Notes Log.}, volume={15}, publisher={ASL}, },
  date={2002},
  pages={92--116},
}

\bib{kohlenbach2}{article}{
  author={Kohlenbach, Ulrich},
  title={Higher order reverse mathematics},
  conference={ title={Reverse mathematics 2001}, },
  book={ series={Lect. Notes Log.}, volume={21}, publisher={ASL}, },
  date={2005},
  pages={281--295},
}

\bib{kohlenbach3}{book}{
  author={Kohlenbach, Ulrich},
  title={Applied proof theory: proof interpretations and their use in mathematics},
  series={Springer Monographs in Mathematics},
  publisher={Springer-Verlag},
  place={Berlin},
  date={2008},
  pages={xx+532},
}

\bib{keuzer}{article}{
  author={Kreuzer, Alexander P.},
  title={The cohesive principle and the Bolzano-Weierstra\ss principle},
  journal={MLQ Math. Log. Q.},
  volume={57},
  date={2011},
  number={3},
  pages={292--298},
}

\bib{kruiserover}{article}{
  author={Kreuzer, Alexander P.},
  title={Primitive recursion and the chain antichain principle},
  journal={Notre Dame J. Form. Log.},
  volume={53},
  date={2012},
  number={2},
  pages={245--265},
}

\bib{kupkaas}{article}{
  author={Kupka, Ivan},
  title={A generalised uniform convergence and Dini's theorem},
  journal={New Zealand J. Math.},
  volume={27},
  date={1998},
  number={1},
  pages={67--72},
}

\bib{christustepaard}{article}{
  author={Li, Gaolin},
  author={Ru, Junren},
  author={Wu, Guohua},
  title={Rudin's lemma and reverse mathematics},
  journal={Ann. Japan Assoc. Philos. Sci.},
  volume={25},
  date={2017},
  pages={57--66},
}

\bib{blindeloef}{article}{
  author={Lindel\"of, Ernst},
  title={Sur Quelques Points De La Th\'eorie Des Ensembles},
  journal={Comptes Rendus},
  date={1903},
  pages={697--700},
}

\bib{medvet}{book}{
  author={Medvedev, Fyodor A.},
  title={Scenes from the history of real functions},
  series={Science Networks. Historical Studies},
  volume={7},
  publisher={Birkh\"auser Verlag, Basel},
  date={1991},
  pages={265},
}

\bib{shoma}{article}{
  author={Montalb\'an, Antonio},
  author={Shore, Richard A.},
  title={The limits of determinacy in second-order arithmetic},
  journal={Proc. Lond. Math. Soc. (3)},
  volume={104},
  date={2012},
  number={2},
  pages={223--252},
}

\bib{montahue}{article}{
  author={Montalb{\'a}n, Antonio},
  title={Open questions in reverse mathematics},
  journal={Bull. Symb. Logic},
  volume={17},
  date={2011},
  number={3},
  pages={431--454},
}

\bib{mooreICM}{article}{
  author={Moore, E. H.},
  title={On a Form of General Analysis with Aplication to Linear Differential and Integral Equations},
  journal={Atti IV Cong. Inter. Mat. (Roma,1908)},
  volume={2},
  date={1909},
  pages={98--114},
}

\bib{moorelimit1}{book}{
  author={Moore, E. H.},
  title={Introduction to a Form of General Analysis},
  journal={The New Haven Mathematical Colloquium},
  publisher={Yale University Press},
  date={1910},
  pages={1--150},
}

\bib{moorelimit2}{article}{
  author={Moore, E. H.},
  journal={Proceedings of the National Academy of Sciences of the United States of America},
  number={12},
  pages={628--632},
  publisher={National Academy of Sciences},
  title={Definition of Limit in General Integral Analysis},
  volume={1},
  year={1915},
}

\bib{moorsmidje}{article}{
  author={Moore, E. H.},
  author={Smith, H.},
  title={A General Theory of Limits},
  journal={Amer. J. Math.},
  volume={44},
  date={1922},
  pages={102--121},
}

\bib{moringpool}{book}{
  author={Moore, E. H.},
  title={General Analysis. Part I. The Algebra of Matrices},
  publisher={Memoirs of the American Philosophical Society, Philadelophia, Vol.\ 1},
  date={1935},
  pages={pp.\ 231 },
}

\bib{mullingitover}{book}{
  author={Muldowney, P.},
  title={A general theory of integration in function spaces, including Wiener and Feynman integration},
  volume={153},
  publisher={Longman Scientific \& Technical, Harlow; John Wiley},
  date={1987},
  pages={viii+115},
}

\bib{mummy}{article}{
  author={Mummert, Carl},
  author={Simpson, Stephen G.},
  title={Reverse mathematics and $\Pi _2^1$ comprehension},
  journal={Bull. Symb. Logic},
  volume={11},
  date={2005},
  number={4},
  pages={526--533},
}

\bib{mummyphd}{book}{
  author={Mummert, Carl},
  title={On the reverse mathematics of general topology},
  note={Thesis (Ph.D.)--The Pennsylvania State University},
  publisher={ProQuest LLC, Ann Arbor, MI},
  date={2005},
  pages={109},
}

\bib{mummymf}{article}{
  author={Mummert, Carl},
  title={Reverse mathematics of MF spaces},
  journal={J. Math. Log.},
  volume={6},
  date={2006},
  number={2},
  pages={203--232},
}

\bib{momg}{article}{
  author={Mummert, Carl},
  author={Stephan, Frank},
  title={Topological aspects of poset spaces},
  journal={Michigan Math. J.},
  volume={59},
  date={2010},
  number={1},
  pages={3--24},
}

\bib{naim}{article}{
  author={Naimpally, Som A.},
  author={Peters, James F.},
  title={Preservation of continuity},
  journal={Sci. Math. Jpn.},
  volume={76},
  date={2013},
  number={2},
}

\bib{noiri}{article}{
  author={Noiri, Takashi},
  title={Sequentially subcontinuous functions},
  year={1975},
  journal={Accad. Naz. dei Lincei},
  volume={58},
  pages={370--373},
}

\bib{dagsam}{article}{
  author={Normann, Dag},
  author={Sanders, Sam},
  title={Nonstandard Analysis, Computability Theory, and their connections},
  journal={To appear in the Journal of Symbolic Logic; arXiv: \url {https://arxiv.org/abs/1702.06556}},
  date={2019},
}

\bib{dagsamII}{article}{
  author={Normann, Dag},
  author={Sanders, Sam},
  title={The strength of compactness in Computability Theory and Nonstandard Analysis},
  journal={To appear in Annals of Pure and Applied Logic; arXiv: \url {http://arxiv.org/abs/1801.08172}},
  date={2019},
}

\bib{dagsamIII}{article}{
  author={Normann, Dag},
  author={Sanders, Sam},
  title={On the mathematical and foundational significance of the uncountable},
  journal={Journal of Mathematical Logic, \url {https://doi.org/10.1142/S0219061319500016}},
  volume={19},
  number={1},
  date={2019},
}

\bib{dagsamV}{article}{
  author={Normann, Dag},
  author={Sanders, Sam},
  title={Pincherle's theorem in Reverse Mathematics and computability theory},
  journal={Submitted, arXiv: \url {https://arxiv.org/abs/1808.09783}},
  date={2018},
}

\bib{dagsamVI}{article}{
  author={Normann, Dag},
  author={Sanders, Sam},
  title={Representations in measure theory},
  journal={Submitted, arXiv: \url {https://arxiv.org/abs/1902.02756}},
  date={2019},
}

\bib{pupu}{article}{
  author={Pu, Pao Ming},
  author={Liu, Ying Ming},
  title={Fuzzy topology. I. Neighborhood structure of a fuzzy point and Moore-Smith convergence},
  journal={J. Math. Anal. Appl.},
  volume={76},
  date={1980},
  number={2},
  pages={571--599},
}

\bib{reedafvegen}{book}{
  author={Reed, Michael},
  author={Simon, Barry},
  title={Methods of modern mathematical physics. I. Functional Analysis},
  publisher={Academic Press},
  date={1981},
  pages={v+400},
}

\bib{manon}{article}{
  author={Riesz, F.},
  title={Sur un th\'eor\`eme de M. Borel},
  journal={Comptes rendus de l'Acad\'emie des Sciences, Paris, Gauthier-Villars},
  volume={140},
  date={1905},
  pages={224--226},
}

\bib{rootsbloodyroots}{article}{
  author={Root, Ralph E.},
  title={Limits in terms of order, with example of limiting element not approachable by a sequence},
  journal={Trans. Amer. Math. Soc.},
  volume={15},
  date={1914},
  number={1},
  pages={51--71},
}

\bib{yamayamaharehare}{article}{
  author={Sakamoto, Nobuyuki},
  author={Yamazaki, Takeshi},
  title={Uniform versions of some axioms of second order arithmetic},
  journal={MLQ Math. Log. Q.},
  volume={50},
  date={2004},
  number={6},
  pages={587--593},
}

\bib{sahotop}{article}{
  author={Sanders, Sam},
  title={Reverse Mathematics of topology: dimension, paracompactness, and splittings},
  year={2018},
  journal={Submitted, arXiv: \url {https://arxiv.org/abs/1808.08785}},
  pages={pp.\ 17},
}

\bib{samsplit}{article}{
  author={Sanders, Sam},
  title={Splittings and disjunctions in Reverse Mathematics},
  year={2019},
  journal={To appear in the Notre Dame Journal for Formal Logic, arXiv: \url {https://arxiv.org/abs/1805.11342}},
  pages={pp.\ 18},
}

\bib{samcie19}{article}{
  author={Sanders, Sam},
  title={Nets and Reverse Mathematics: initial results},
  year={2019},
  journal={LNCS 11558, Proceedings of CiE19, Springer},
  pages={pp.\ 12},
}

\bib{samwollic19}{article}{
  author={Sanders, Sam},
  title={Reverse Mathematics and computability theory of domain theory},
  year={2019},
  journal={LNCS 11541, Proceedings of WoLLIC19, Springer},
  pages={pp.\ 20},
}

\bib{samph}{article}{
  author={Sanders, Sam},
  title={Bootstraps, nets, and hierarchies},
  year={2019},
  journal={Submitted, arxiv: \url {https://arxiv.org/abs/1908.05676}},
  pages={pp.\ 29},
}

\bib{zot}{book}{
  author={Schechter, Eric},
  title={Handbook of analysis and its foundations},
  publisher={Academic Press, Inc., San Diego, CA},
  date={1997},
  pages={xxii+883},
}

\bib{simpson1}{collection}{
  title={Reverse mathematics 2001},
  series={Lecture Notes in Logic},
  volume={21},
  editor={Simpson, Stephen G.},
  publisher={ASL},
  place={La Jolla, CA},
  date={2005},
  pages={x+401},
}

\bib{simpson2}{book}{
  author={Simpson, Stephen G.},
  title={Subsystems of second order arithmetic},
  series={Perspectives in Logic},
  edition={2},
  publisher={CUP},
  date={2009},
  pages={xvi+444},
}

\bib{sigohi}{incollection}{
  author={Simpson, Stephen G.},
  title={{The G\"odel hierarchy and reverse mathematics.}},
  booktitle={{Kurt G\"odel. Essays for his centennial}},
  pages={109--127},
  year={2010},
  publisher={Cambridge University Press},
}

\bib{steengoed}{book}{
  author={Steen, L.},
  author={Seebach, J.},
  title={Counterexamples in topology},
  publisher={Dover},
  date={1995},
  pages={xii+244},
}

\bib{stillebron}{book}{
  author={Stillwell, J.},
  title={Reverse mathematics, proofs from the inside out},
  pages={xiii + 182},
  year={2018},
  publisher={Princeton Univ.\ Press},
}

\bib{zwette}{book}{
  author={Swartz, Charles},
  title={Introduction to gauge integrals},
  publisher={World Scientific},
  date={2001},
  pages={x+157},
}

\bib{thom2}{book}{
  author={Thomson, B.},
  author={Bruckner, J.},
  author={Bruckner, A.},
  title={Elementary real analysis},
  publisher={Prentice Hall},
  date={2001},
  pages={pp.\ 740},
}

\bib{demofte}{article}{
  author={Timofte, Vlad},
  author={Timofte, Aida},
  title={Generalized Dini theorems for nets of functions on arbitrary sets},
  journal={Positivity},
  volume={20},
  date={2016},
  number={1},
  pages={171--185},
}

\bib{tomaat}{article}{
  author={Toma, V.},
  title={Strong convergence and Dini theorems for non-uniform spaces},
  journal={Ann. Math. Blaise Pascal},
  volume={4},
  date={1997},
  number={2},
  pages={97--102},
}

\bib{troelstra1}{book}{
  author={Troelstra, Anne Sjerp},
  title={Metamathematical investigation of intuitionistic arithmetic and analysis},
  note={Lecture Notes in Mathematics, Vol.\ 344},
  publisher={Springer Berlin},
  date={1973},
  pages={xv+485},
}

\bib{troeleke1}{book}{
  author={Troelstra, Anne Sjerp},
  author={van Dalen, Dirk},
  title={Constructivism in mathematics. Vol. I},
  series={Studies in Logic and the Foundations of Mathematics},
  volume={121},
  publisher={North-Holland},
  date={1988},
  pages={xx+342+XIV},
}

\bib{kliet}{article}{
  author={Vietoris, Leopold},
  title={Stetige Mengen},
  language={German},
  journal={Monatsh. Math. Phys.},
  volume={31},
  date={1921},
  number={1},
  pages={173--204},
}

\bib{nieuweman}{article}{
  author={von Neumann, John},
  title={On complete topological spaces},
  journal={Trans. Amer. Math. Soc.},
  volume={37},
  date={1935},
  number={1},
  pages={1--20},
}

\bib{wangjoke}{article}{
  author={Wang, Hao},
  title={Eighty years of foundational studies},
  journal={Dialectica},
  volume={12},
  date={1958},
  pages={466--497},
}

\bib{wolk}{article}{
  author={Wolk, E. S.},
  title={Continuous convergence in partially ordered sets},
  journal={General Topology and Appl.},
  volume={5},
  date={1975},
  number={3},
  pages={221--234},
}

\bib{youngster}{article}{
  author={Young, W. H.},
  title={Overlapping intervals},
  journal={Bulletin of the London Mathematical Society},
  date={1902},
  volume={35},
  pages={384-388},
}

\bib{zadeh65}{article}{
  author={Zadeh, L. A.},
  title={Fuzzy sets},
  journal={Information and Control},
  volume={8},
  date={1965},
  pages={338--353},
}


\end{biblist}
\end{bibdiv}
\bye